\documentclass[sn-mathphys,Numbered]{sn-jnl}

\usepackage{graphicx}%
\usepackage{multirow}%
\usepackage{amsmath,amssymb,amsfonts}%
\usepackage{amsthm}%
\usepackage{mathrsfs}%
\usepackage[title]{appendix}%
\usepackage{xcolor}%
\usepackage{textcomp}%
\usepackage{manyfoot}%
\usepackage{booktabs}%
\usepackage{algorithm}%
\usepackage{algorithmicx}%
\usepackage{algpseudocode}%
\usepackage{listings}%
\usepackage{bm}
\usepackage{algpseudocode}
\usepackage{lipsum}
\usepackage{lineno,dsfont}
\usepackage{amscd}
\usepackage{stmaryrd}
\usepackage{color}
\usepackage{ifthen}
\usepackage{subfigure}
\usepackage{epsfig,wrapfig,epic}
\usepackage{url}
\usepackage{dsfont}
\usepackage[none]{hyphenat}
\usepackage{hyperref}

\newcommand{\rank}{\mathrm{rank}}

\newcommand{\Jac}{\mathcal{J}}
\newcommand{\Sing}{\mathfrak{S}}

\newcommand{\PDE}{\sc{pde}}
\newcommand{\ODE}{\sc{ode}}
\newcommand{\DAE}{\sc{dae}}
\newcommand{\PDAE}{\sc{pdae}}

\newcommand{\AP}{\sc{ap}}
\newcommand{\randpoint}{\ensuremath{\mathfrak{a}}}

\newcommand{\R}{\mathbb{R}}

\newcommand{\ie}{{i}.{e}.}
\newcommand{\eg}{{e}.{g}.\hspace{1pt}}
\newcommand{\etc}{{etc}.}
\newcommand{\LD}{\textsc{ld}}

\def\ctotDer{\textbf{D}}

\newtheorem{example}{Example}[section]
\newtheorem{remark}{Remark}[section]
\newtheorem{prop}{Proposition}[section]

\newtheorem{define}{Definition}[section]
\newtheorem{theorem}{Theorem}[section]
\newtheorem{lemma}{Lemma}[section]

\makeatletter
\newenvironment{breakablealgorithm}
{
	\begin{center}
		\refstepcounter{algorithm}
		\hrule height.8pt depth0pt \kern2pt
		\renewcommand{\caption}[2][\relax]{
			{\raggedright\textbf{\ALG@name~\thealgorithm} ##2\par}%
			\ifx\relax##1\relax 
			\addcontentsline{loa}{algorithm}{\protect\numberline{\thealgorithm}##2}%
			\else 
			\addcontentsline{loa}{algorithm}{\protect\numberline{\thealgorithm}##1}%
			\fi
			\kern2pt\hrule\kern2pt
		}
	}{
		\kern2pt\hrule\relax
	\end{center}
}
\makeatother

\begin{document}

\title[Embedding Method by Real Numerical Algebraic Geometry  for Structurally Unamenable Differential-Algebraic Equations]{Embedding Method by Real Numerical Algebraic Geometry  for Structurally Unamenable Differential-Algebraic Equations}


\author[1,2]{\fnm{Wenqiang} \sur{Yang}}\email{yangwenqiang@cigit.ac.cn}

\author*[1]{\fnm{Wenyuan} \sur{Wu}}\email{wuwenyuan@cigit.ac.cn}

\author[3]{\fnm{Greg} \sur{Reid}}\email{reid@uwo.ca}

\affil*[1]{\orgdiv{Chongqing Key Laboratory of Automated Reasoning and Cognition}, \orgname{Chongqing Institute of Green and Intelligent Technology, Chinese Academy of Sciences}, \orgaddress{\street{266 Fangzheng Avenue}, \city{Beibei District}, \postcode{400714}, \state{Chongqing}, \country{China}}}

\affil[2]{\orgdiv{Chongqing School}, \orgname{University of Chinese Academy of Sciences}, \orgaddress{\street{266 Fangzheng Avenue}, \city{Beibei District}, \postcode{400714}, \state{Chongqing}, \country{China}}}

\affil[3]{\orgdiv{Department of Mathematics}, \orgname{University of Western Ontario}, \orgaddress{\street{1151 Richmond Street}, \city{London}, \postcode{N6A5B7}, \state{Ontario}, \country{Canada}}}


\abstract{
Existing structural analysis methods may fail to find all hidden constraints for a system of differential-algebraic equations with parameters if the system is structurally unamenable for certain values of the parameters.

In this paper, for polynomial systems of differential-algebraic equations, numerical methods are given to solve such cases using numerical real algebraic geometry. First, we propose an embedding method that for a given real analytic system constructs an equivalent system with a full-rank Jacobian matrix. Secondly, we introduce a witness point method, which can help to detect degeneration on all components of constraints of such systems. Thirdly, the two methods above lead to a numerical global structural analysis method for structurally unamenable differential-algebraic equations on all components of constraints.
}


\maketitle
\section{Introduction}\label{sec:intro}
\sloppy{}

Systems of differential-algebraic equations are widely used to model and simulate dynamical systems such as mechanical systems,
electrical circuits, and chemical reaction plants {\cite{Ilchmann15}}. In applications, many models are naturally real and nonlinear, including polynomial systems and analytic systems.
We will often use the abbreviation
{\DAE}\footnote{A system of differential-algebraic equations will be denoted by {\DAE} while {\DAE}s will denote several such systems.}
for a system of differential-algebraic equations.
The name arose since such systems usually contain differential equations with derivatives and algebraic
equations without derivatives.
It was initially falsely believed that any such {\DAE} could be easily
converted by coordinate changes and eliminations to a traditional explicit {\ODE} --- the so-called
underlying {\ODE}\footnote{A system of explicit ordinary differential equations,
in solved form for their highest derivatives, will be denoted by {\ODE}.}.

The distance between a {\DAE} and its explicit {\ODE} can be measured by
a \textbf{differential index} {\cite{Campbell95,Campbell1995}} which is the minimum number of differentiations required to transform the {\DAE} into its corresponding explicit {\ODE} and is used in our paper. See Definition  \ref{define_index} for the definition of differential index. \textbf{Index reduction} is a differentiation operation {\cite{Shampine02}} to convert a {\DAE} into an explicit {\ODE} by reducing the differential index.



One may try to numerically solve a {\DAE} directly without reducing its index {\cite{LIU201593,
Awawdeh09,Pryce98}}
. However, this is only feasible for low index (index $\leq 1$) problems.
In contrast to the above direct approach, indirect and widely used approaches first use
index reduction only by differentiation \cite{Pryce98,Pryce01,Pantelides88,Gear88,Fritzson14} followed
by consistent initial point determination \cite{Pantelides88,Brenan95,Taihei19}. A point satisfying the \textbf{constraints} of a {\DAE} is called a  \textbf{consistent point} if an unique solution of the {\DAE} exists through this point. See Definition  \ref{define_ConsistentPoint} and Equation \ref{eq:cons} for the definition of consistent point and constraints. In this paper, we make contributions to such indirect approaches.

Note that the name {\DAE} misleadingly suggests that a {\DAE} can be partitioned into differential equations and non-differential equations (algebraic equations) where the latter
are regarded as constraints. But constraints may have lower order derivatives as shown in Example \ref{point}.
\begin{example}\label{point}
  Consider the {\DAE}
\begin{equation}
  u' u''  +  u u' + x = 0, ((u')^2 + u^2 + x^2 - 1)((u')^2 + u^2 + x^2 - 4 ) = 0
\end{equation}
where $u$ is a unknown function of $x$. $ u' u''  +  u u' + x = 0$ is a differential equation.
Then $((u')^2 + u^2 + x^2 - 1)((u')^2 + u^2 + x^2 - 4 ) = 0$ is a constraint even though it contains derivatives.
Geometrically there are $2$ \textbf{component of constraints} (spheres of radius $1$ and $2$). See Definition \ref{de:zero_set} for the definition of component of constraints.  
\end{example}

\subsection{Previous Work}\label{ssec:pre-work}


A {\DAE} may have more than one component of constraints.
Finding at least one consistent point on each component of constraints of a {\DAE} is an important problem \cite{Pantelides88}.
Commonly used methods to obtain such a consistent initial point are the approximation method \cite{Leimkuhler91,Shampine02} and the transformation method \cite{Vieira20011,Brown98}. 
After that, the index reduction method can give an equivalent system in the neighborhood of the consistent point. As it described in \cite{Tuomela1998}, any {\DAE} can be regarded as sub-manifolds of jet bundles. In this way, Tuomela \cite{Tuomela1998} gave an index reduction method such that consistent points on all component of constraints for a polynomial {\DAE} can be found through algebraic ideals and Gr\"{o}bner base techniques in jet space \cite{saunders_1989}. But these techniques have extremely high complexity. Moreover, Gr\"{o}bner base is available only for polynomial {\DAE}s instead of the general case of analytic {\DAE}s. 



Hence, there are more studies focusing on index reduction in the neighborhood of a consistent point.
Gear {\cite{{Gear88}} proposed a method based on repeatedly finding algebraic equations could realize the index reduction after sufficiently many differentiations of a {\DAE}. 
However, all of these methods are notoriously hard for large and non-linear systems. Fortunately there are some efficient structural analysis methods, such as Pantelides' bipartite graph method \cite{Pantelides88} and Mattsson-S\"{o}derlind's dummy derivatives method {\cite{Mattsson93,McKenzie17}}, based on bipartite graph preprocessing that can sometimes reduce the differential index.
Pryce {\cite{Pryce01}} gave a more general direct and widely applicable $\Sigma$-method by solving the dual problem of an assignment problem ({\AP}). A {\DAE} is called \textbf{structurally amenable} (S-amenable) {\cite{Nedialkov2022}} if any of above structural analysis methods succeed on it. An important recent contribution by Nedialkov, Pryce and Sholz \cite{Nedialkov2022} shows how a wide class of electrical circuit systems yield index $\leq 1$ S-amenable systems.

 Despite the success of structural analysis, the methods may fail  when its  \textbf{Jacobian} matrix after differentiation is singular. Such {\DAE}s are called \textbf{structurally unamenable} (S-unamenable) {\DAE}s.
 It is essential to develop improved structural analysis methods to transform an S-unamenable {\DAE} into an S-amenable {\DAE}\footnote{ The successful transformation by an improved structural analysis method is called \textbf{regularization}.}. See $\bm{\Jac}_{k_c}$ in Equation (\ref{Jac}) for the definition of Jacobian.

  Direct methods such as symbolic elimination {\cite{Campbell93,Piipponen2014}} and substitution method \cite{Taihei19}, can regularize a {\DAE} in theory but they are very complex and inefficient for a nonlinear {\DAE}.  Linear {\DAE} with constant coefficients can be transformed into the canonical form of Weierstra\ss, the Kronecker index determined and then the {\DAE} can be solved directly {\cite{Gerdts11,Kunkel2006}}.

  There are also many indirect methods based on a ``combinatorial relaxation" algorithm 
  {\cite{Murota95}. The LC-method of Tan et al.\ \cite{Tan17} considers equations and their derivatives, with better results for some nonlinear {\DAE}s. The ES-method \cite{Tan17} uses new variables to seek the solution in a projection of a higher dimensional space, and it can be considered as a supplement for the LC-method.
In order to avoid high complexity of eliminations, the augmentation method \cite{Taihei19} adopts a principle similar to the ES-method. However, the transformation of the above methods is feasible only in a neighborhood of a consistent point and is not applicable for some general cases as shown in Section \ref{ssec:problem}.



\subsection{Problem Statement and Contributions}\label{ssec:problem}

In this paper,  S-unamenable {\DAE}s are called ``\textbf{degeneration}" cases, including two types: \textbf{symbolic cancellation} (see Example \ref{ex:2}) and \textbf{numerical degeneration} (see Example \ref{ex:3}).


\begin{example}\label{ex:2} Symbolic Cancellation: {Example $4.10$ in \cite{Tan17}}: 

\[ \left\{ \begin{array}{cclcl} f_1&=&{x}_{1}+t{x}_{2}+t^{2}{x}_{3}+g_1(t)&=& 0\\ f_2&=&{x}'_{1}+t{x}'_{2}+t^2{x}'_{3}+g_2(t)&=&0\\ f_3&=&{x}''_{1}+t{x}''_{2}+2t^2{x}''_{3}+g_3(t)&=& 0 \end{array} \right . \rightarrow \bm{\Jac}_{k_c}=\begin{array}{cc} & \begin{array}{rcl} {x}''_{1}&{x}''_{2}& {x}''_{3} \end{array}\\ \begin{array}{c} f''_1\\f'_2\\f_3\end{array} & \left(\begin{array}{ccc} 1~&t&~t^2 \\1~&t&~t^2 \\1~&t&~2t^2\end{array}\right)\end{array}\]

The determinant of the Jacobian matrix $\bm{\Jac}_{k_c}$ of this {\DAE} after the $\Sigma$-method is identically zero. It implies that this {\DAE} is a typical symbolic cancellation case.
Furthermore, this case can be regularized by a number of existing methods.
\end{example}


Unfortunately, there is little research on failure caused by numerical degeneration. This could happen for a parametric {\DAE} model with a non-zero determinant, where parameters take some specific values, and the determinant equals zero after substituting any initial value on a component of constraints. 

\begin{figure}[htpb]
  \centering
  \includegraphics[height=2cm,width=6cm]{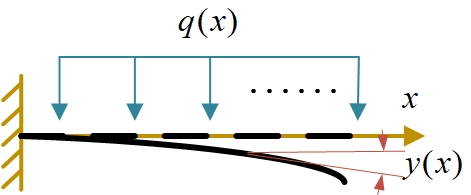}
   \caption{The Superposition Deformation of Beams under Collinear Bending Moments }
  \label{fig:bending}
\end{figure}

\begin{example}\label{ex:3}Numerical Degeneration:

On a beam shown in Figure \ref{fig:bending}, when a load $q(x)$ distributed in the $x$-axis is applied, a bending moment $M(x)=q(x)\cdot x^2/2$ is generated at any point $x$.
As it is described in {\cite{Rans}}, the bending deformation $y(x)$ of the beam satisfies $EI\cdot{\frac {\rm d^{2}}{{\rm d}x^{2}}}y(x)=M(x)$, where $E$ and $I$ are given elastic modulus and moment of inertia of an area of the beam. When two collinear loads act on the beam at the same time, the deformation is the superposition of the effects of these loads. Usually, sensor can only measure the sum of these two loads $q_1(x)+q_2(x)$. Suppose that the input energy $W(x)=\frac{1}{2}\cdot E\cdot y(x)^2$ of two loads at each point $x$ on the beam is linear in $\lambda$, $\lambda W_1(x)=W_2(x)$, then the deformation variables $y_1(x)$, $y_2(x)$ satisfy:
\[\left\{\begin{array}{rcc}
{\frac {\rm d^{2}}{{\rm d}x^{2}}}y_{1}(x)+{\frac {\rm d^{2}}{{\rm d}x^{2}}}y_{2}(x)-\frac{x^2}{EI}\cdot(q_1(x)+q_2(x))&=&0\\
\lambda\cdot y_{1}(x)^{2}-y_{2}(x)^{2}&=&0
\end{array}\right.\]
\[\rightarrow \bm{\Jac}_{k_c}=\left(\begin{array}{cc} 1 & 1 \\ 2\lambda\cdot y_1 & -2y_2 \end{array}\right)\]

In this example, the determinant of the Jacobian matrix of $(\frac {\rm d^{2}}{{\rm d}x^{2}}y_{1},\frac {\rm d^{2}}{{\rm d}x^{2}}y_{2})^{T}$ is $-2(y_{2}+\lambda y_{1})$. When the parameter $\lambda$ is equal to $1$, the constraint becomes $y_{1}^{2}-y_{2}^{2}=(y_{1}+y_{2})(y_{1}-y_{2})=0$. Physically, this means that the elastic deformation energy of each point on the beam is the same. Obviously, two consistent initial points can be selected from the two different components of the constraints, respectively. If the point is on the component of constraints $y_1-y_2=0$, then the $\Sigma$-method works well. But
otherwise, we always encounter a singular Jacobian matrix, and we call this case numerical degeneration.
\end{example}

Such degeneration is of potential importance in designing control parameters in architecture, aviation and biochemistry \cite{Biegler,Kunkel2018}.
So a key question deserving further study both theoretically and computationally addressed in our paper is how to identify and solve such S-unamenable systems.




To sum up, in this paper, our goal is to find an efficient way to regularize numerical degeneration {\DAE}s over each component of
constraint. To approach it, we propose two essential methods in this paper:
\begin{itemize}
  \item   For a given real analytic system, whether it is symbolic cancellation or numerical degeneration, we give an embedding method to handle it. The embedding method eventually converts an S-unamenable {\DAE} into an equivalent S-amenable {\DAE} without algebraic elimination, or reports failure which implies the system has no solution. 
  \item  We present a global numerical approach based on real algebraic geometry which finds at least one point on each component of constraints of a polynomial {\DAE} and detect its degeneration.
\end{itemize}

\section{Preliminaries}\label{s:pre}
\sloppy{}

In what follows we will use algorithmic aspects of the formal (Jet) theory of differential equations {\cite{Reid01,Seiler2010}}.  Jet theory enables two inter-related views of differential equations to be algorithmically and correctly manipulated.  One view is in terms of the maps as algebraic (non-differential) equations, and the other view is in terms of local solutions of the differential equations.

 Let an independent variable $ t \in \mathbb{I} \subseteq \R$
and suppose dependent variables
$\bm{x},\bm{x}^{(1)},...,\bm{x}^{(\ell)}$ are vectors in $\R^n$, where $\mathbb{I}$ is a closed interval and $\ell$ is a fixed positive integer.
Here we consider maps
$\bm{F}: \mathbb{I}\times \R^{\ell n+n}\rightarrow \R^n$ which are real analytic
in $\bm{x},\bm{x}^{(1)},...,\bm{x}^{(\ell)}$ and $t$.

Our approach exploits certain properties of the solutions sets (zero sets)
of systems of real analytic equations.



\begin{define}[Real zero set, Component]\label{de:zero_set}
The real zero set of a real analytic system $\bm f= \bm 0$ is denoted by $Z_{\R}(\bm f)$. If $Z_{\R}(\bm f) =\bigcup_{i \in \mathcal{I}} C_i$ where each $C_i$ is a connected component of $Z_{\R}(\bm f)$ of the system and is a manifold, we call $C_i$ a \textbf{component of} $Z_{\R}(\bm f)$. If $Z_{\R}(\bm f)$ is the zero set of the constraints of a real analytic {\DAE}, then $C_i$ is called a \textbf{component of the constraints}.
\end{define}

\begin{example}
The constraint of Example \ref{point} can be decomposed as the union of analytic manifolds (connected components) --- spheres of radius $1$ and $2$ in $(x,u,u')$-space. Each of them is a component of $Z_{\R}(\bm f)$.
   
\end{example}

In the case of polynomial systems the index set $\mathcal{I}$ is finite, and can be represented by computable points (witness points) lying on each component.  In the real analytic case $\mathcal{I}$ can be infinite.
Note that here we use components as is standard terminology in algebraic geometry and differential geometry, but differs from terminology used in other areas, where a component is the coordinate of a vector for example.

In particular, we consider systems which are not solved for their highest derivatives, and regard such a system as a
{\DAE}.
The differential-algebraic systems we consider have Jet form
\begin{equation}\label{eq:def_dea}
\bm{F}(t,\bm{x},\bm{x}^{(1)},...,\bm{x}^{(\ell)}) = 0.
\end{equation}

In \cite{Pryce01}, Pryce reinterpreted Pantelides' algorithm \cite{Pantelides88} as the dual problem of an assignment problem (AP) that reveals structural information about {\DAE}s.

Firstly, we can obtain an $n \times n$ matrix
$\bm \sigma=(\sigma_{i,j})_{1 \leq i \leq n,1 \leq j \leq n}$ which is called the \textbf{signature matrix}
 of $\bm{F}$ by Pryce \cite{Pryce01}:
\begin{equation}\label{signature}
    (\sigma_{i,j})(\bm{F}) = \left\{%
\begin{array}{ll}
    \hbox{ the order of $\LD(F_i,x_j)$;} \\
    -\infty, \;\; \hbox{if $x_j$ does not occur in $F_i$.} \\
\end{array}%
\right.
\end{equation}

The \textbf{leading derivative} of an equation or a system $F_i$ with respect to
$x_j$ is denoted by $\LD(F_i,x_j)$ and is $x_j^{(k)}$ occurring in $F_i$ such that some $k\in \mathds{Z^{+}}$ is the highest order of derivative of $x_j$ in $\bm{F}$.

 The \textbf{bipartite graph} of $\bm{F}$ is the undirected graph whose $2n$ vertices are the $n$ row set of equations and $n$ column set of variables, and which has an edge between row $i$ and column $j$ whenever $(i,j)\in (\sigma_{i,j})(\bm{F})$. We say the bipartite graph of $\bm{F}$ has a \textbf{perfect matching} if there exists an edge subset $M$ such that every vertex of the graph is incident to exactly one edge of $M$.

Secondly, the $\Sigma$-method \cite{Pryce01} can be formulated as the dual problem of an
 {\AP} problem with \textbf{optimal solution} $\bm{c} =
(c_1, ..., c_n )$ and $\bm{d} = (d_1, ..., d_n )$ of integer vectors:
\begin{equation}\label{LPP}
\begin{array}{l}
    \hbox{Minimize~~} \delta(\bm{F}) = \sum d_j - \sum c_i,  \\
    \hbox{~~~~where~~} d_j -c_i \geq \sigma_{ij},  \\
    ~~~~~~~~~~~~~~c_i \geq 0. \\
\end{array}%
\end{equation}

Here, $\delta(\bm{F})$ is the \textbf{optimal value} of the Problem (\ref{LPP}).

Let $\ctotDer$ be the formal total derivative operator with respect to independent variable $t$, $\ctotDer = \frac{\partial}{\partial t} +  \sum_{k=0}^{\infty} \bm{x}^{(k+1)} \frac{\partial }{\partial \bm{x}^{(k)}}.$

We specify the differentiation order for $F_i$ to be $c_i$, for $i = 1, \dots, n$. Then the differentiation of $\bm{F}$  up to the order of optimal solution of $\bm{c}$ is
\begin{equation}\label{eq:DefFc}
\bm{F}^{(\bm{c})}= \{F_1, \ctotDer F_1,..., \ctotDer^{c_1}F_1\}\cup \cdots \cup  \{F_n, \ctotDer F_n,..., \ctotDer^{c_n}F_n\} =\ctotDer^{\bm{c}}\bm{F}.
\end{equation}
 The number of equations of \textbf{differentiated system} $\bm{F}^{(\bm{c})}$ is $n + \sum_{i=1}^{n} c_i$.


We note that $\bm{F}^{(\bm{c})}$ has a favorable block
triangular structure enabling us to compute consistent initial values
more efficiently.

Without loss of generality, we
assume $c_1 \geq c_2 \geq \cdots \geq c_n$, and let $k_c = c_1$,
then we can
partition $\bm{F}^{(\bm{c})}$ into $k_c + 1$ parts (see Table $1$),
for $0\leq p \leq k_c$ with $p \in\mathds{Z}$ given by
\begin{equation}\label{eq:B_i}
    \bm{B}_p= \{ F_j^{(p+c_j - k_c )} : 1\leq j \leq n, p+c_j - k_c
    \geq 0 \} .
\end{equation}

\begin{table}\label{table:triB}
\begin{tabular}{|c|c|c|c|c|}
  \hline
  $\bm{B}_0$ & $\bm{B}_1$ & $\cdots$ & $\bm{B}_{k_c-1}$ & $\bm{B}_{k_c}$ \\
  \hline
  $F^{(0)}_1$ & $F^{(1)}_1$  & $\cdots$      & $F^{(c_1-1)}_1$  & $F^{(c_1)}_1$  \\
            & $F^{(0)}_2$  & $\cdots$      & $F^{(c_2-1)}_2$  & $F^{(c_2)}_2$ \\
            &            & $\vdots$      & $\vdots$         & $\vdots$  \\
            &            & $F^{(0)}_n$   &  $\cdots$        & $F^{(c_n)}_n$ \\
  \hline
\end{tabular}
\caption {The triangular block structure of $\bm{F}^{(\bm{c})}$ for the case of $c_{p}=c_{p+1}+1$; For $0\leq p < k_c$, $\bm{B}_i$ has fewer dependent variables than $\bm{B}_{p+1}$.}
\end{table}

Here, the \textbf{top block} of $\bm{F}^{(\bm{c})}$ is $\bm{B}_{k_c}$ and the \textbf{constraints} are
\begin{equation}\label{eq:cons}
\bm{F}^{(\bm{c}-1)} = \{\bm{B}_0,...,\bm{B}_{k_c-1}\} .
\end{equation}

Similarly, let $k_d = \max(d_j)$ and we can partition all the variables into $k_d+1$ parts:
$\bm{X}^{(q)}= \{ x_j^{(q+d_j - k_d )} : 1\leq j \leq n, q+ d_j - k_d  \geq 0 \}$ .

For each $\bm{B}_i, 0\leq i \leq k_c$, we define the \textbf{Jacobian Matrix}
\begin{equation}\label{Jac}
    \bm{\Jac}_{i} =
\left( \frac{\partial \bm{B}_{i}}{\partial \bm{X}^{(i+k_d-k_c)}} \right).
\end{equation}

So $\bm{\Jac}_{k_c}$ is the Jacobian Matrix of the top block in the table, and it is a square matrix.

\begin{prop}\label{prop:fullrank}
Let $ \{\bm{\Jac}_i\}$ be the set of Jacobian
matrices of $\{\bm{B}_i \}$. For any $0 \leq i < j \leq k_c$, $\bm{\Jac}_{i}$
is a sub-matrix of $\bm{\Jac}_{j}$. Moreover, if $\bm{\Jac}_{k_c}$ has full
rank, then any $\bm{\Jac}_{i}$ also has full rank.
\end{prop}

See \cite{WRI09} for the proof.

\begin{define}\label{define_index}
The \textbf{differential index} of an $\ell$ th-order {\DAE}  in Equation (\ref{eq:def_dea}) is the minimum number of differentiations of any equation needs to be done. After this differentiations, the first-order {\DAE} is obtained by replacing higher-order derivatives of $\bm{x}$ by newly introduced variables.
\end{define}

\begin{define}\label{define_ConsistentPoint}
A point satisfying the constraints of a {\DAE} is called a \textbf{consistent point} if a unique solution of the {\DAE} exists through this point.
\end{define}
We have expressed the definition of consistent point to reflect its more general relative, that of locally solvability for general systems of PDE (or {\PDAE}). See Reid et. al \cite{Reid01} for a jet theoretic introduction focused on {\DAE} and {\PDAE}.  \\
Remarkably, differentiation of systems of differential equations to include missing constraints until existence and uniqueness results could be derived, first arose around 1900 in the case of linear systems of {\PDE} in the works of Cartan, Riquier and others \cite{Reid01}.  The finiteness of this process was conjectured by Cartan and eventually proved by Kuranishni \cite{Seiler2010}. The number of required differentiations is related to various differential indices for {\DAE}. Our article is focused on extending structural methods for {\DAE}.  When such methods are successful, they yield a differential index counting the number of required differentiations.  Our focus in this paper is not on differential indices, but rather on extending the applicability of structural methods.

Suppose $(t^*,\bm{X}^*)$ is a consistent point satisfying the constraints $\{\bm{B}_0,...,\bm{B}_{k_c-1}\}$ and
$\bm{\Jac}_{k_c}$ has full rank at this point, then the $\Sigma$-method has successfully finished the index reduction.

In the rest of the paper, we usually suppress the subscript in $\bm{\Jac}_{k_c}$ so it becomes $\bm{\Jac}$ unless the subscript is needed.

\begin{define}[Degree of freedom (DOF) for {\DAE}s]\label{def:DOF}
Let a {\DAE} $\bm{F}$ contains $m$ equations and $n$ dependent variables, DOF of $\bm{F}$ is $n - m$ which determines the existence of the solution.
\end{define}

 The optimal value found by the $\Sigma$-method can be regarded as DOF measure for an S-amenable {\DAE}
. When a {\DAE} is S-unamenable, its true DOF is overestimated by $\delta(\bm{F})$ after the $\Sigma$-method. Roughly speaking, finding hidden constraints is equivalent to minimizing $\delta(\bm{F})$.

Unfortunately, the definition of optimal value $\delta(\bm{F})$ is limited to square systems, and we need to extend the definition for non-square systems  $\bm{F}^{(\bm{c})}$.

\begin{define}\label{define_delta1}
Let a differentiated system $\bm{F}^{(\bm{c})}$ consist of the top block $\bm{B}_{k_c}$ and the remaining equations $\bm{F}^{(\bm{c}-1)}$,  where $\bm{F}^{(\bm{c})}$ contains $p$ equations and $n$  dependent variables with $p\geq n$. Let $\delta(\bm{B}_{k_c})$ be the optimal value of the dual problem of the {\AP} of $\bm{B}_{k_c}$'s signature matrix. We define  $\delta(\bm{F}^{(\bm{c})}) = \delta(\bm{B}_{k_c}) - \#eqns(\bm{F}^{(\bm{c}-1)})$, where $\#eqns(\bm{F}^{(\bm{c}-1)})$ is the number of equations in $\bm{F}^{(\bm{c}-1)}$.  Meanwhile, $\delta(\bm{F}^{(\bm{c})})$  also equals the $DOF$ measure of $\bm{F}^{(\bm{c})}$ if $\bm{B}_{k_c}$ is non-singular.
\end{define}

For example, in the case of $\bm{c}=\bm{0}$, we have $\#eqns(\bm{N})=0$ since  $p=n$, and our definition of $\delta(\bm F)$ is equivalent to the original definition in Equation (\ref{LPP}).

\begin{prop}\label{prop:extend_delta}
Let $(\bm{c},\bm{d})$ be an optimal solution of Problem (\ref{LPP}) for a given  DAE $\bm{F}$. Then $\delta (\bm{F})=\delta (\bm{F}^{(\bm{c})})=\sum d_j - \sum c_i$.

\end{prop}

\begin{proof}
For a differentiated {\DAE} system $\bm{F}^{(\bm{c})}=\{\bm{B}_{k_c},\bm{F}^{(\bm{c}-1)}\}$, the signature matrix of the top block $\bm{B}_{k_c}$ is square.

 We construct a pair $(\hat{\bm{c}},\hat{\bm{d})}$, for $i=1,\cdots,n$ and $ j=1,\cdots,n$, $\hat{c}_{i}=0$ and $\hat{d}_{j}= d_{j}$. Since $(\bm{c},\bm{d})$ is the optimal solution for $\bm{F}$, and $\bm{B}_{k_c}$ is the top block of $\bm{F}^{(\bm{c})}$, it follows that $(\hat{\bm{c}},\hat{\bm{d})}$ is the optimal solution of $\bm{B}_{k_c}$, $\delta(\bm{B}_{k_c})= \sum d_j$.

 By Definition \ref{define_delta1},  with $\#eqns(\bm{F}^{(\bm{c}-1)})=\sum c_i$, we can obtain $$\delta (\bm{F}^{(\bm{c})})= \delta(\bm{B}_{k_c}) -\#eqns(\bm{F}^{(\bm{c}-1)})=\sum d_j - \sum c_i =\delta (\bm{F}) \,.$$
\end{proof}
\section{Detecting Degeneration by Points}\label{sec:det}



There are two types of methods to detect a nonzero function vanishing on a component of $Z_{\R}(\bm f)$: the symbolic method and the numerical method.
The symbolic methods, {\eg}\ by using Gr\"{o}bner Bases \cite{Geddes1992} or Triangular Decomposition \cite{Collins1982}, have high complexity and are only feasible for polynomial cases. The numerical method, {\eg}\ by using Newton iterative solvers, can find some points on the components of $Z_{\R}(\bm f)$. But using these points to detect the vanishing of the function is incorrect for smooth functions. Therefore, we need to introduce an efficient numerical method to detect the vanishing of the function.
To build a solid foundation of our theory, we need some results from the theory of real analytic functions of several variables \cite{KrantzParks02}.

\begin{define}
A function $f$, with domain an open subset $U \subset \R^n$ and range
$\R$, is called real analytic on $U$, if for each $\bm p \in U$ the function
$f$ can be represented by a convergent power series in some neighborhood of $\bm p$.
\end{define}

\begin{prop}[Proposition 2.2.8 of \cite{KrantzParks02}]\label{prop:composition}
If $f_1, ..., f_m$ are real analytic in some neighborhood of
the point $\bm p\in \R^n$  and $g$ is real analytic in some neighbourhood of the point
$(f_1(\bm p),  ... , f_m(\bm p)) \in \R^m$, then the composition of functions $g(f_1(\bm x), ...,f_m(\bm x))$ is real analytic
in a neighborhood of $\bm p$.
\end{prop}

\begin{theorem}[Real Analytic Implicit Function Theorem \cite{KrantzParks02}]\label{thm:IFT}
Given a set of equations
$f_i(x_1,...,x_m;y_1,...,y_n) =0, \;\; i=1,2,...,n$, where each $f_i$ is real analytic, suppose that
$(\bm p,\bm q)= (p_1,...,p_m; q_1,...,q_n)$  is a solution with nonsingular Jacobian
$\frac{\partial \bm f}{\partial \bm y} (\bm p,\bm q)$.

Then there exists a neighborhood $U \subset \R^m$ of $\bm p$ and a set of real analytic functions $ \phi_j : U \rightarrow  \R, j = 1, 2, ... , n$, such that $\phi_j(\bm p) = q_j, j = 1,2,...,n$, and
$$f_i(\bm x; \phi_1(\bm x),...,\phi_n(\bm x)) =0, i = l,2,...,n, $$
hold for $\bm x \in U$.
\end{theorem}

\begin{theorem}[Identity Theorem for Real Analytic Functions]\label{thm:IT}
Given two real analytic functions $f$ and $g$ on an open and connected set $U\subset \R^n$, if $f = g$ on a nonempty open subset $S  \subseteq U$, then $f = g$ on the whole set $U$.
\end{theorem}
See \cite{lebl2019tasty} for a proof.



\begin{lemma}\label{lem:measure0}
Let $Z_{\R}(\bm F)$ be the zero set of a real analytic system,  let $C$ be a component of $Z_{\R}(\bm F)$ in $\R^{m+n}$ of dimension $m$ and let $\bm{f}$ be a real analytic function on $\R^{m+n}$.
Then the intersection $C\cap Z_{\R}(\bm{f})$ is equal to $C$ or has measure zero over $C$.
\end{lemma}
\begin{proof}
Since $C$ is a connected component described implicitly by a real analytic system, by the Implicit Function Theorem \ref{thm:IFT}, locally $C$ can be parameterized by $m$ free coordinates. To be rigorous, we need an atlas over
$C$ which is a collection of charts depending on the free variables, with a countable open covering $C = \bigcup_i U_i$.
 Without loss of generality, we assume the parameterization is $y_j = \phi_j(x_1,...,x_m),  j=1,...,n$, where $\phi_j$ is real analytic.

By Proposition $1$ of \cite{Mityagin}, if $\bm{f}$ is a real analytic function on a connected open domain $U_i$ of $\R^{m}$, and $\bm{f}$ is not  identically zero, then $U_i\cap Z_{\R}(\bm{f})$ has measure zero. There are two possible cases: 1) $U_i\cap Z_{\R}(\bm{f})$ has measure zero for each $U_i$; or 2) $\bm f$ is identically zero over some $U_k$.
For case one, we have $C \cap Z_{\R}(\bm{f}) = \bigcup_i (U_i \cap Z_{\R}(\bm{f}))$  still has measure zero.
Otherwise,  there is a nonempty open subset $U_k$ of $C$, where $g(\bm x) = \bm f(\bm x,\phi_1(\bm x),...,\phi_n(\bm x)) = \bm 0$. By Proposition \ref{prop:composition}, $\bm{g}$ is real analytic. Due to the Identity Theorem \ref{thm:IT}, $\bm{g} = \bm{0}$ on the whole component, and thus $C \subseteq Z_{\R}(\bm{f})$.


\end{proof}

Consider a component of constraints $C$ of a real analytic {\DAE}  $\bm{F}^{(\bm{c})}$ with a real point $\bm p \in \R^n$.
Suppose $\rank \bm{\Jac}(\bm p) = r < n$. Without loss of generality, we assume that the sub-matrix $\bm{\Jac}(\bm p)[1\:r,1\:r]$ has full rank.
In the following, we will show that the rank is almost a constant over $C$.

\begin{lemma}\label{lem:whole}
Let $Z_{\R}(\bm F)$ be the zero set of a real analytic system, $C$ be a component of $Z_{\R}(\bm F)$. If $\bm{\Jac}[1\:r,1\:r]$ has full rank at some point $\bm p$ on $C$, then it is non-singular almost everywhere on $C$. 
\end{lemma}
\begin{proof}
Let $f$ be the determinant of $\bm{\Jac}[1\:r,1\:r]$.
If $f(\bm p) \neq 0$, then $C \nsubseteq Z_{\R}(f)$ and  Lemma \ref{lem:measure0} implies that $C\cap Z_{\R}(f)$ has measure zero over $C$.
\end{proof}

\begin{remark}\label{rem:constant_rank}
Let $f$ be the determinant of $\bm{\Jac}$. In Lemma \ref{lem:measure0}, consider a component $C$ of $Z_{\R}(\bm F)$.  If $C$ and the zero set of the determinant of $\bm{\Jac}$ do not intersect, then $\bm{\Jac}$ has full rank over $C$. Otherwise, there are $2$ cases of the intersection $C\cap Z_{\R}(f)$: (1) it is equal to $C$, ($2$) it has measure zero on $C$. 

Case (1) is the focus of our paper. Then the rank $r$ can be determined by using the real witness set $W$.  Please see more details in Section \ref{sec:poly}. 

For case (2), the $\bm{\Jac}$ has full rank almost everywhere on $C$. There must be a singular set on $C$, which is denoted
by $\Sing$ is those points at which $Z_{\R}(f)$ intersects with $C$. The singular set of $f= \bm 0$, denoted by $\Sing$ is those points at which $Z_{\R}(f)$ is locally not a real analytic manifold. On one hand, the $\bm{\Jac}$ has full rank over $C\backslash \Sing$. On the other hand, at a point of $\Sing$, the rank of $\bm{\Jac}$ is less than $r$, it is also a constant at this point. Our embedding method can also handle it, but it is also quite complex. Due to length limitations, to solve {\DAE}s on the singular set is outside the scope of our paper.
\end{remark}

Thus, in this paper we suppose Jacobians have constant rank over a whole component of contraunts. Jacobians with constant rank enable us to embed the zero set into a higher dimensional space.

\begin{lemma}[Constant Rank Embedding]\label{lem:proj}
Let
$$\bm f= \{f_1(x,y,z),...,f_r(x,y,z)\} \hspace{0.4cm} \mbox{and} \hspace{0.4cm} \bm g= \{g_1(x,y,z),...,g_{n-r}(x,y,z)\} $$
be two sets of real analytic functions, where
$\bm x=(x_1,...,x_r)$, $\bm y=(y_1,...,y_{n-r})$ and $\bm z=(z_1,...,z_m)$. Let $Z_{\R}(\bm F)$ be the zero set of a real analytic system, $C$ be a component of $Z_{\R}(\bm F)$ in $\R^{m+n}$.
If the Jacobian matrices $\frac{\partial (\bm f,\bm g)}{\partial (\bm x, \bm y)}$ and $\frac{\partial \bm f}{\partial \bm x}$ have constant rank $r$ on $C$,
then $$Z_{\R}(\bm f,\bm g) \cap C = \pi Z_{\R}(\bm f(\bm x,\bm y,\bm z), \bm f(\bm u,\bm \xi,\bm z),\bm g(\bm u,\bm \xi,\bm z)) \cap C$$ holds, where $\bm u=(u_1,...,u_r)$ and $\bm \xi$ is a constant vector in $\R^{n-r}$ and $\pi$ is the projection from $(\bm x,\bm y,\bm z,\bm u)$-space to $(\bm x,\bm y,\bm z)$-space.
\end{lemma}
\begin{proof} \hskip4pt
Since  $\frac{\partial \bm f}{\partial \bm x}$ has constant rank $r$ on $C$, by the Implicit Function Theorem \ref{thm:IFT} and the Identity Theorem \ref{thm:IT}, there exists a set of real analytic functions $\bm \phi = \{\phi_1,...,\phi_r\}$  such that
$\bm f(\bm \phi(\bm y,\bm z),\bm y,\bm z)=\bm 0$ for any $(\bm y,\bm z) \in \pi_{\bm{yz}}(C)$, where $\pi_{\bm{yz}}$ is the
projection from $(\bm x,\bm y,\bm z,\bm u)$-space to $(\bm y,\bm z)$-space. Thus, $$\frac{\partial \bm f}{\partial \bm x} \frac{\partial \bm \phi}{\partial \bm y} + \frac{\partial \bm f}{\partial \bm y} =\bm 0\,. $$

Since $\frac{\partial (\bm f,\bm g)}{\partial (\bm x, \bm y)}$ also has constant rank $r$, $(\frac{\partial \bm \phi}{\partial \bm y}, I)^t$  is in the
null-space of $\frac{\partial (\bm f,\bm g)}{\partial (\bm x, \bm y)}$. So $$\frac{\partial \bm g}{\partial \bm x} \frac{\partial \bm \phi}{\partial \bm y} + \frac{\partial \bm g}{\partial \bm y} =\bm 0\,. $$

Let $\bm G(\bm y,\bm z) = \bm g(\bm \phi(\bm y,\bm z),\bm y,\bm z)$. We have $\frac{\partial \bm G}{\partial \bm y} = \frac{\partial \bm g}{\partial \bm x} \frac{\partial \bm \phi}{\partial \bm y} + \frac{\partial \bm g}{\partial \bm y} =\bm 0$, which implies that $\bm G(\bm y,\bm z) = \bm G(\bm \xi,\bm z)$ for any constant $\bm \xi \in \R^{n-r}$.

If $\bm p=(\bm p_{\bm{x}},\bm p_{\bm{y}},\bm p_{\bm{z}}) \in Z_{\R}(\bm f,\bm g) \cap C$, then $\bm p_{\bm{x}}= \bm \phi(\bm p_{\bm{y}},\bm p_{\bm{z}})$. Let $\bm p_{\bm{u}}= \bm \phi(\bm \xi,\bm p_{\bm{z}})$ for some constant vector $\bm \xi$, and let $\hat{\bm p}= (\bm p_{\bm{x}},\bm p_{\bm{y}},\bm p_{\bm{z}},\bm p_{\bm{u}})$. It is straightforward to verify that $\bm f(\bm p_{\bm{u}},\bm \xi,\bm p_{\bm{z}})=\bm 0$ and $\bm g(\bm p_{\bm{u}},\bm \xi,\bm p_{\bm{z}})= \bm G(\bm \xi,\bm p_{\bm{z}}) =  \bm G(\bm p_{\bm{y}},\bm p_{\bm{z}}) = \bm g(\bm p_{\bm{x}},\bm p_{\bm{y}},\bm p_{\bm{z}}) = \bm 0$.  Therefore, $\hat{\bm p} \in Z_{\R}(\bm f(\bm x,\bm y,\bm z), \bm f(\bm u,\bm \xi,\bm z),\bm g(\bm u,\bm \xi,\bm z))$. Thus $\bm p \in \pi Z_{\R}(\bm f(\bm x,\bm y,\bm z), \bm f(\bm u,\bm \xi,\bm z),\bm g(\bm u,\bm \xi,\bm z)) \cap C$, proving $Z_{\R}(\bm f,\bm g) \cap C \subseteq \pi Z_{\R}(\bm f(\bm x,\bm y,\bm z), \bm f(\bm u,\bm \xi,\bm z),\bm g(\bm u,\bm \xi,\bm z)) \cap C$.

For any  $\bm p=(\bm p_{\bm{x}},\bm p_{\bm{y}},\bm p_{\bm{z}}) \in \pi Z_{\R}(\bm f(\bm x,\bm y,\bm z), \bm f(\bm u,\bm \xi,\bm z),\bm g(\bm u,\bm \xi,\bm z)) \cap C$, we have $\bm p_{\bm{x}}= \bm \phi(\bm p_{\bm{y}},\bm p_{\bm{z}})$ and ${\bm{u}}= \bm \phi(\bm \xi,\bm p_{\bm{z}})$.
Also $\bm g(\bm u,\bm \xi,\bm p_{\bm{z}})=\bm 0 \Rightarrow \bm 0=\bm G(\bm \xi,\bm p_{\bm{z}})=\bm G(\bm p_{\bm{y}},\bm p_{\bm{z}})= \bm g(\bm p_{\bm{x}},\bm p_{\bm{y}},\bm p_{\bm{z}})$.
So $\bm p\in Z_{\R}(\bm f,\bm g)\cap C$, proving $Z_{\R}(\bm f,\bm g) \cap C \supseteq \pi Z_{\R}(\bm f(\bm x,\bm y,\bm z), \bm f(\bm u,\bm \xi,\bm z),\bm g(\bm u,\bm \xi,\bm z)) \cap C$.

\end{proof}

\section{The Embedding Method for S-unamenable  {\DAE}s}\label{sec:implictit method}
In this section, we mainly focuses on the  combinatorial relaxation framework {\cite{Murota95}.
We construct a new {\DAE} ${\bm{G}}$ with the same solution space of the {\DAE} $\bm{F}$ and associate optimal values $\delta(\bm{F})$ and $\delta({\bm{G}})$ satisfying $0\leq\delta({\bm{G}})<\delta(\bm{F})$. 

If we have an initial point, then according to Lemma {\ref{lem:whole}}, then the rank of Jacobian matrix
on almost everywhere of the corresponding component of constraints of the {\DAE} can be calculated by singular value decomposition (SVD) given by \cite{Golub13}.

By Remark \ref{rem:constant_rank}, suppose $Z_{\R}(\bm{F}^{(\bm{c})})$ has constant rank {\ie}
\begin{equation}\label{eq:rank}
\rank \bm{\Jac} = r =\rank \bm{\Jac}[1\:r,1\:r]  < n
\end{equation}
over a component of constraints $C$ of $Z_{\R}(\bm{F}^{(\bm{c-1})})$. To simplify our description, we specify the full rank submatrix to be $\bm{\Jac}[1\:r,1\:r]$, which always can be done by proper permutations of variables and equations given by Algorithm \ref{alg:4} (see \ref{sec:alg}).

\begin{define}\label{define_IRE}
The embedding Method: Suppose $(\bm{c},\bm{d})$ is the optimal solution of Problem (\ref{LPP}) for a given {\DAE} $\bm{F}$, and then differentiated {\DAE} $\bm{F}^{(\bm{c})} = \{  \bm{B}_{k_c},\bm{F}^{(\bm{c}-1)}\}$ has constant rank $\rank \bm{\Jac} = r < n$. Let $\bm{s}= (x_1^{(d_1)}, ..., x_r^{(d_r)})$, $\bm{y}= (x_{r+1}^{(d_{r+1})},...,x_n^{(d_n)})$ and $\bm{z}= (t,\bm{X},\bm{X}^{(1)},...,\bm{X}^{(k_d-1)})$, then $\bm{B}_{k_c}=\{ \bm{f(s,y,z)},\bm{g(s,y,z)}\}$, where $\bm{f(s,y,z)}=\{F_1^{(c_1)},..., F_r^{(c_r)}\}$ with full rank Jacobian and $\bm{g(s,y,z)}=\{F_{r+1}^{(c_{r+1})},..., F_n^{(c_n)}\}$.
 Then $\bm{G}$ is constructed by the following steps:

\begin{enumerate}
  \item  Introduce $n$ new equations $\hat{\bm{F}}=\{\bm{f(u,\xi,z)},\bm{g(u,\xi,z)}\}$:
to replace $\bm{s}$ in the top block $\bm{B}_{k_c}$ by $r$ new dependent variables $\bm{u}=(u_1,...,u_r)$  respectively, and simultaneously replace $\bm{y}$ in the top block $\bm{B}_{k_c}$ by $n-r$ constants $\bm{\xi}=(\xi_1,...,\xi_{n-r})\in \R^{n-r}$ respectively.
  \item Construct a new square subsystem
\begin{equation}\label{eq:newsys}
\bm{F}^{\rm aug} = \{ \bm{f(s,y,z)},\bm{\hat{F}(s,y,z)}\}.
\end{equation}
\item Construct $\bm{G}= \{\bm{F}^{\rm aug} ,\bm{F}^{(\bm{c}-1)}\}$.
\end{enumerate}
\end{define}

where $\bm{F}^{\rm aug}$ has $n+r$ equations with $n+r$ leading variables $\{\bm{X}^{(k_d)}, \bm{u}\}$ and $\bm{X}^{(k_d)}=\{\bm{s},\bm{y}\}$.

Since this reduction step introduces a new variable $\bm{u}$, the corresponding lifting of the consistent initial values must be addressed.
One approach to this problem is to solve the new system $\bm{F}^{\rm aug}$ to obtain lifted consistent initial values.
But this approach is unnecessary and expensive.
According to Definition \ref{define_IRE}, the consistent initial values of the new variables $\bm{u}$ can simply be taken as the initial values of their replaced variables $\bm{s}$. Then $\bm{\xi}$ takes the same initial value as was assigned to $\bm{y}$.

By the way, for the case ($2$) in Remark {\ref{rem:constant_rank}}, the initial values of the singular point is also beneficial for the success of the embedding method.


\begin{theorem}\label{thm:result}
Let $(\bm{c},\bm{d})$ be the optimal solution of Problem (\ref{LPP}) for a given {\DAE} $\bm{F}$.
Let $\bm{F}^{(\bm{c})} = \{ \bm{B}_{k_c}, \bm{F}^{(\bm{c}-1)} \}$ as defined in Equation (\ref{eq:B_i}). If $\bm{F}^{(\bm{c})}$ satisfies (\ref{eq:rank}), and $C$ is a component of constraints of $Z_{\R}(\bm{F}^{(\bm{c-1})})$ in $\R^{\sum d_j+n}$, then
$$Z_{\R}(\bm{F}^{(\bm{c})})\cap C = \pi Z_{\R}(\bm{G}) \cap C$$
where $\bm{G}= \{\bm{F}^{\rm aug},\bm{F}^{(\bm{c}-1)} \}$ as defined in Definition \ref{define_IRE}.   
 Moreover, we have $\delta(\bm{G}) \leq  \delta(\bm{F}) - (n-r)$.
\end{theorem}
\begin{proof}
By the Constant Rank Embedding Lemma \ref{lem:proj} and Definition \ref{define_IRE}, since the constants involved can be arbitrarily ascribed, we easily get $Z_{\R}(\bm{B}_{k_c}) \cap C =  \pi Z_{\R}(\bm{F}^{\rm aug}) \cap C$.  Further, since
$\bm{F}^{(\bm{c}-1)}$ is common to
both ${\bm{F}^{(\bm{c})}}$ and ${\bm{G}}$ we have $Z_{\R}(\bm{F}^{(\bm{c})})\cap C = \pi Z_{\R}(\bm{G}) \cap C$.

\begin{table}
\caption{signature matrix of $\bm{F}^{\rm aug}$ and a feasible pair of $(\bar{\bm{c}},\bar{\bm{d}})$}\label{table:tritop}
\renewcommand\arraystretch{1.2}
\begin{tabular}{cccc}
&&\multicolumn{1}{c}{ $\bar{\bm{d}}(1,...,n)=\bm{d}\quad $}  &   \multicolumn{1}{c}{ $\quad \bar{\bm{d}}(n+1,...,n+r)=\bm{1}$ } \cr\cmidrule{2-4}
  &\multicolumn{1}{|c|}{$\sigma_{i,j}(\bm{F}^{\rm aug})$}& \multicolumn{1}{|c|}{$\bm{x}$ }& \multicolumn{1}{|c|}{$\bm{u}$} \cr\cmidrule{2-4}
   \multicolumn{1}{l|}{$\bar{\bm{c}}(1,...,r)= \bm{0}$} &\multicolumn{1}{|c|}{$\bm{f(s,y,z)}$} & \multicolumn{1}{|c|}{$\sigma_{i,j}\leq d_{j}$} & \multicolumn{1}{|c|}{$\sigma_{i,j}=-\infty$} \cr\cmidrule{2-4}
   \multicolumn{1}{l|}{$\bar{\bm{c}}(r+1,...,n+r)= \bm{1}$ } & \multicolumn{1}{|c|}{$\hat{\bm{F}}(\bm{z},\bm{\xi},\bm{u})$} &  \multicolumn{1}{|c|}{$\sigma_{i,j}\leq d_{j}-1$}     &  \multicolumn{1}{|c|}{$\sigma_{i,j}=0$ or $-\infty$} \cr\cmidrule{2-4}
\end{tabular}
\end{table}

According to Table \ref{table:tritop}, we construct a pair  $(\bar{\bm{c}},\bar{\bm{d})}$:
\begin{equation}\label{opt_soln}
 \bar{c}_{i}= \left\{%
\begin{array}{ll}
    0,& i=1,\cdots,r;\\
    1,& i=(r+1),\cdots,(n+r). \\
\end{array}%
\right., \bar{d}_{j}= \left\{%
\begin{array}{ll}
    d_{j},& j=1,\cdots,n; \\
    1,& j=(n+1),\cdots,(n+r). \\
\end{array}%
\right.
\end{equation}

For $1\leq i \leq r$ and $ 1\leq j \leq n$, the signature matrix of $\bm{F}^{\rm aug}$ is the same as $\bm{B}_{k_c}[1:r,1:n]$, implying that $\sigma_{i,j}(\bm{f})\leq d_{j}-0=\bar{d}_{j}-\bar{c}_{i}$.

For $1\leq i \leq r$ and  $ (n+1)\leq j \leq (n+r)$, $\sigma_{i,j}(\bm{f})=-\infty<1-0=\bar{d}_{j}-\bar{c}_{i}$.

 For $(r+1)\leq i \leq (n+r)$ and $ 1\leq j \leq n$, since $\bm{s}$ and $\bm{y}$ in $\hat{\bm{F}}$ have been replaced with dummy variables and constants, we have:  \\
 $\sigma_{i,j}(\hat{\bm{F}})\leq \sigma_{i,j}(\bm{B}_{k_c})-1 \leq d_{j}-1= \bar{d}_{j}-\bar{c}_{i}$

  For $(r+1)\leq i \leq (n+r)$ and $(n+1)\leq j \leq (n+r)$,  $\sigma_{i,j}(\hat{\bm{F}})\leq 0 = \bar{d}_{j}-\bar{c}_{i}$.

 To sum up, $(\bar{\bm{c}},\bar{\bm{d})}$ is feasible pair of the Problem (\ref{LPP}) for $\bm{F}^{\rm aug}$. Thus, $\delta(\bm{F}^{\rm aug})\leq \sum\limits^{n+r}_{j=1}\bar{d}_{j}-\sum\limits^{n+r}_{j=1} \bar{c}_{i} =\sum\limits^{n}_{j=1}{d}_{j}-(n-r) = \delta(\bm{B}_{k_c}) -(n-r) $.



Obviously, since both ${\bm{F}^{(\bm{c})}}$ and ${\bm{G}}$ have the same
block of constraints $\bm{F}^{(\bm{c}-1)}$, according to Definition {\ref{define_delta1}}, it follows that $\delta({\bm{G}})-\delta({\bm{F}^{(\bm{c})}})=\delta(\bm{F}^{\rm aug})-\delta(\bm{B}_{k_c})\leq -(n-r)$. Finally,  $\delta(\bm{G}) \leq \delta(\bm{F}^{(\bm{c})}) - (n-r)= \delta(\bm{F}) - (n-r)$, since $\delta(\bm{F})=\delta(\bm{F}^{(\bm{c})})$ by Proposition \ref{prop:extend_delta}.
\end{proof}


Since $\rank \bm{\Jac} = r$, restoring regularity is equivalent is some sense to finding $n-r$ hidden constraints by elimination.

Although there are more dependent variables in $\bm{G}$, the computational cost is much lower than explicit symbolic elimination, since $\bm{G}$ and the corresponding lifted initial value points can be easily constructed.

Moreover, since the special structure of $\bm{G}$, in the embedding method, the feasible pair $(\bar{\bm{c}},\bar{\bm{d})}$  given in Equation (\ref{opt_soln}) without solving Problem (\ref{LPP}) is an optimal solution in all examples in Section \ref{sec:ex}. Theoretically, Lemma {\ref{lem:lifting}} below
shows that the feasible pair $(\bar{\bm{c}},\bar{\bm{d})}$ is optimal under some reasonable assumptions.

\begin{lemma}\label{lem:lifting}
Suppose each equation $F_{i}$ in the top block $\bm{B}_{k_c}$ of a {\DAE} $\bm F$
contains at least one variable  $x_{j}\in \bm{X}^{(k_d)-1}$.  If the bipartite graph of $\bm F$ has a perfect matching, then $(\bar{\bm{c}},\bar{\bm{d})}$ in Equation (\ref{opt_soln}) is an optimal solution and $\delta(\bm{G}) =  \delta(\bm{F}) - (n-r)$.
\end{lemma}
\begin{proof}
According to the Table \ref{table:tritop}, since $\bm{f(s,y,z)}$ is a part of {\DAE} $\bm{F}$, its corresponding $(\bar{\bm{c}}[1:r],\bar{\bm{d}}[1:n])$ is optimal. If $(\bar{\bm{c}},\bar{\bm{d}})$ is not an optimal solution, then there must be a feasible pair $({\bm{c}},{\bm{d})}$ satisfies one of the following four cases, such that $(\sum \bm{d}- \sum {\bm{c}})\leq \sum (\bar{\bm{d}}- \sum \bar{\bm{c}})$. The Lemma is now proved by contradiction.

(1) $\bm{c}=\bar{\bm{c}}$ and at least one element in ${\bm{d}}[(n + 1):(n + r)]$ is $0$. It is easy to prove it does not satisfy $d_j-c_i\geq \sigma_{i,j}$.

(2) $\bm{d}=\bar{\bm{d}}$ and at least one element in ${\bm{c}}[(r + 1):(n + r)]$ is more than $1$. It also does not satisfy $d_j-c_i\geq \sigma_{i,j}$.

(3) Some elements in ${\bm{c}}[(r + 1):(n + r)]$ are zeros and more elements in ${\bm{d}}[(n + 1):(n + r)]$ are also zeros. That implies, at least $2$ of the highest derivatives of variables in $\bm{u}$ only occur in $1$ equation.  This contradicts the perfect matching condition.

(4) Some elements in ${\bm{c}}[(r + 1):(n + r)]$ are $>1$, and some elements in ${\bm{d}}[(n + 1):(n + r)]$ are also $>1$ , such that $\sum {\bm{d}}[(n + 1):(n + r)]-\sum {\bm{c}}[(r + 1):(n + r)]< (r-n)$. Since at least one of $\bm{X}^{(k_d)-1}$ occurs in every equation of $\hat{\bm{F}}(\bm{z},\bm{\xi},\bm{u})\}$, all elements in ${\bm{c}}[(r + 1):(n + r)]$ must be $\leq 1$, contrary to our assumption.

Since $(\bar{\bm{c}},\bar{\bm{d})}$ is an optimal solution of Problem (\ref{LPP}) for $\bm{F}^{\rm aug}$ it follows that,
$\delta(\bm{F}^{\rm aug}) =  \delta(\bm{B}_{k_c}) -(n-r) \Rightarrow \delta(\bm{G}) =  \delta(\bm{F}) - (n-r)$.
\end{proof}

\section{Global Structural Analysis Method For Polynomial {\DAE}s}\label{sec:poly}

 {\DAE}s may have more than one components of constraints (see Example {\ref{ex:3}}). The goal of global structural analysis methods is to construct corresponding $\bm{G}$ for every component of constraints to restore regularity for S-unamenable {\DAE}s.

Theoretically, if there is a solver can find at least one consistent initial point on
each component of constraints, the global structural analysis method performs well. Indeed real analytic {\DAE}s may have infinitely many components of constraints, unlike polynomial {\DAE}. Moreover, since the homotopy continuation methods can only provide all zero sets of a square polynomial system, we just consider polynomial {\DAE}s in this section. Therefore, we propose a global structural analysis method based on the embedding method and real algebraic geometry to detect and regularize S-unamenable cases
without using symbolic methods, such as determinants or Gr\"{o}bner bases.





Numerical algebraic geometry~\cite{SVW05,Hauenstein2017} was pioneered by Sommese, Wampler, Verschelde and others for describing positive dimensional zero sets of polynomial systems (see \cite{BHSW13,SommeseWampler05} for references and background). The approach is built on \textbf{witness points} which arise by slicing the zero sets of a polynomial system (see \cite{CoxLittleOshea07} for definitions) with appropriate arbitrary planes of complementary dimension. These complex witness points can be efficiently computed by homotopy continuation solvers \cite{Lee2008}, and are theoretically guaranteed to compute at least one such point on each component of the zero sets of the
 polynomial system.

For the real case, the methods in \cite{WuReid13,WRF17} yield real witness points as critical points of the distance from a arbitrary hyperplane to the real analytic manifold.
Alternatively,  the real witness points can be considered as critical points of the distance from an arbitrary point to the real analytic manifold \cite{Hauenstein12}.

%
%
%
More precisely, to solve a polynomial system $ \bm f = \{f_1,...,f_k\}$ with variables $\{x_1,...,x_n\}$, we first choose an arbitrary point $\randpoint \in \R^n$ (for numerical stability, we usually choose a random point $\randpoint$ on the unit sphere of $\R^n$).  Then there is at least one point on each  component of $Z_{\R}(\bm f)$ with minimal distance to $\randpoint$ satisfying the following problem:
\begin{eqnarray}\label{eq:opt2}
 \min \;  \sum_{i=1}^n (x_i-\randpoint_i)^2/2 \nonumber \\
 s.t. \hspace{1cm} \bm f(\bm x) = 0   \nonumber.
\end{eqnarray}

This optimization problem can be formulated as a square system by using Lagrange multipliers, i.e. $\bm g = \{\bm f,  \sum_{i=1}^k \lambda_i \nabla f_i + x_i - \randpoint_i \} = 0.$

When $\bm f$ satisfies the regularity assumptions in \cite{WuReid13}, all the real zero sets of $\bm g = \bm 0$ can be obtained by
the homotopy continuation method.
These points are called real witness points of $Z_{\R}(\bm f)$, where $Z_{\R}(\bm f)=\{\bm x\in \mathbb{R}^{n} : \bm f(\bm x)=0\}$. Although the homotopy continuation  solver \cite{BHSW13} is based on floating-point arithmetic, it can theoretically guarantee to obtain real zero sets with sufficient accuracy \cite{BCL2013}.
These real zero sets of the constraint equations provide initial points for every component of constraints of a {\DAE}.

\begin{define}\label{def5.1}
For a polynomial system $\bm f$, if a finite set $W\subseteq \R^n$ contains at least one point on each  component of $Z_{\R}(\bm f)$, then $W$ is called a \textbf{real witness set} of $Z_{\R}(\bm f)$ and the points of $W$ are called \textbf{real witness points}. So $W$ is not unique.
\end{define}




\begin{theorem}\label{thm:vanish}
Let $\bm g$ be a polynomial system and $W$ be a real witness set of $Z_{\R}(\bm g)$. If another polynomial system $\bm f(\bm p) = 0$ for any $\bm p\in W$, then $Z_{\R}(\bm g) \subseteq Z_{\R}(\bm f)$ for almost all choices of $W$.
\end{theorem}
\begin{proof}
According to Definition \ref{def5.1}, $Z_{\R}(\bm g)$ consists of finitely many  components of $Z_{\R}(\bm g)$.
By Lemma \ref{lem:measure0}, for each component of $Z_{\R}(\bm g)$, the intersection $(Z_{\R}(\bm g)) \cap  Z_{\R}(\bm f)$ has measure zero over
$Z_{\R}(\bm g)$, unless $(Z_{\R}(\bm g)) \subseteq Z_{\R}(\bm f)$. 

Since $\bm p$ is an arbitrary point on $Z_{\R}(\bm g)$, it does not belong to the measure zero set $(Z_{\R}(\bm g)) \cap  Z_{\R}(\bm f)$ for almost all choices of $\bm{p}$. Therefore, $Z_{\R}(\bm g) \subseteq Z_{\R}(\bm f)$ for almost all choices of $W$. Since $Z_{\R}(\bm f)$ is a closed set, the closure of $Z_{\R}(\bm g)$, which is $Z_{\R}(\bm g)$, must be contained in $Z_{\R}(\bm f)$.
\end{proof}

In summary, Nedialkov and Pryce \cite{Nedialkov2008} already provided a DAETS code for solving {\DAE} initial value problems efficiently and stably by Taylor series expansion for S-amenable {\DAE}s. Our global structural analysis method enable such approaches to be extended to S-unamenable {\DAE}.  Our approach ( see Algorithm \ref{alg:7} in \ref{sec:alg}) can be described as follows: Firstly, we can find a witness set $W$ of the constraints by homotopy continuation methods. After we obtain a witness set $W$, it is unnecessary to compute the determinant of $\bm{\Jac}$ during the detection of S-unamenable cases by Theorem \ref{thm:vanish}.
We can simply substitute $W$ into the Jacobian matrix and compute its numerical rank $r$ by SVD. According to Lemma \ref{lem:whole} and Theorem \ref{thm:vanish}, we can obtain the rank of the Jacobian matrix on all components of constraints. Finally, we apply the embedding method to regularize this {\DAE} on each component of constraints.

\section{Examples}\label{sec:ex}
In this section, we give three examples. These include two symbolic cancellation examples: a transistor amplifier \cite{Mazzia08}, and a modified pendulum \cite{Mazzia08}. 
Also included is one numerical degeneration example: the bending deformation of a beam (Example \ref{ex:3}).

In particular, we apply the following four methods to the above $3$ {\DAE}s: 
(a) the $\Sigma$-method, 
(b) the substitution method, (c) the augmentation method, and (d) the embedding method (see Appendix \ref{sec:alg} for algorithms). Here, the code of method (b) and method (c) are available in the Git repository established at \url{https://github.com/OptMist-Tokyo/DAEPreprocessingToolbox}.


\subsection{Transistor Amplifier}\label{ssec:exam1}
%
%

First, we discuss a transistor amplifier example from an electrical network. It is a linear {\ODE} system with an identically singular Jacobian matrix:
\[\left\{\begin{array}{rcc}
C_1\cdot(\dot{x}_{1}-\dot{x}_{2})+(x_{1}-U_{e})/R_{0}&=&0\\
C_1\cdot(\dot{x}_{1}-\dot{x}_{2})-(1-\alpha)\cdot f(x_{2}-x_{3})+U_b/R_2-x_{2}\cdot(1/R_1+1/R_2)&=&0\\
C_2\cdot\dot{x}_{3}+x_{3}/R_3-f(x_{2}-x_{3})&=&0\\
C_3\cdot(\dot{x}_{4}-\dot{x}_{5})+x_{4}/R_4-U_b/R_4+\alpha\cdot f(x_{2}-x_{3})&=&0\\
C_3\cdot(\dot{x}_{4}-\dot{x}_{5})-x_{5}\cdot(1/R_5+1/R_6)+U_b/R_6-(1-\alpha)\cdot f(x_{5}-x_{6})&=&0\\
C_4\cdot\dot{x}_{6}+x_6/R_7-f(x_{5}-x_{6})&=&0\\
C_5\cdot(\dot{x}_{7}-\dot{x}_{8})+x_{7}/R_8-U_b/R_8+\alpha\cdot f(x_{5}-x_{6})&=&0\\
C_5\cdot(\dot{x}_{7}-\dot{x}_{8})-x_{8}/R9&=&0
\end{array}\right.\]

where $f(x)=\beta(exp(x/U_F)-1)$ and $U_{e}= 0.1\cdot sin(200\pi t)$ with constant parameters $U_b$, $U_F$, $\alpha$, $\beta$, $R_0, R_1, \ldots, R_9$, and $C_1, \ldots, C_5$. For more details, see \cite{Mazzia08}.

The structural information obtained by the $\sigma$-method is that the dual optimal solution is $\bm{c}=\bm{0}_{1\times 8}$ and $\bm{d}=\bm{1}_{1\times 8}$, such that $n=\delta=8$. For the Jacobian matrix,  we have $\rank \bm{\Jac} = r =\rank \bm{\Jac}[(4,6,3,1,7),(4,6,3,1,7)]= 5$.

Obviously, we still cannot solve the system directly after the $\Sigma$-method. Fortunately, as it is a linear {\DAE}, almost all existing improved structural methods can be used to regularize it.

It is easy to get $\bm{F}=\bm{F}^{(\bm{c})}$ since $\bm{c}$ is a zero vector.
By the embedding method, according to Definition {\ref{define_IRE}}, we have $\bm{s}=\{\dot{x}_{4}, \dot{x}_{6}, \dot{x}_{3}, \dot{x}_{1}, \dot{x}_{7}\}$, $\bm{y}=\{\dot{x}_{2}, \dot{x}_{5}, \dot{x}_{8}\}$, $\bm{f(s,y,z)}=\{F_{4}, F_6, F_3, F_1, F_7\}$ and $\bm{g(s,y,z)}=\{F_{2}, F_5, F_8\}$. Thus, $\hat{\bm{F}}=\{\bm{f(u,\xi,z)},\bm{g(u,\xi,z)}\}$,
where $\bm{s}$ and $\bm{y}$ are replaced by $(u_1, u_2, u_3, u_4, u_5)$ and some constants $(\xi_1, \xi_2, \xi_3)\in \R^{3}$ respectively.  Finally, we construct a new top block $\bm{F}^{\rm aug} = \{\bm{f(s,y,z)},\hat{\bm{F}}\}$ of the differentiated {\DAE}, where $\hat{\bm{F}}$ is given below.


\[\hat{\bm{F}}=\left\{\begin{array}{rcc}
C_1\cdot(u_4-\xi_1)+(x_{1}-U_{e})/R_{0}&=&0\\
C_1\cdot(u_4-\xi_1)-(1-\alpha)\cdot f(x_{2}-x_{3})+U_b/R_2-x_{2}\cdot(1/R_1+1/R_2)&=&0\\
C_2\cdot u_3+x_{3}/R_3-f(x_{2}-x_{3})&=&0\\
C_3\cdot(u_1-\xi_2)+x_{4}/R_4-U_b/R_4+\alpha\cdot f(x_{2}-x_{3})&=&0\\
C_3\cdot(u_1-\xi_2)-x_{5}\cdot(1/R_5+1/R_6)+U_b/R_6-(1-\alpha)\cdot f(x_{5}-x_{6})&=&0\\
C_4\cdot u_2+x_6/R_7-f(x_{5}-x_{6})&=&0\\
C_5\cdot(u_5-\xi_3)+x_{7}/R_8-U_b/R_8+\alpha\cdot f(x_{5}-x_{6})&=&0\\
C_5\cdot(u_5-\xi_3)-x_{8}/R9&=&0
\end{array}\right.\]

After the embedding method, we can directly construct an optimal solution of the dual problem of {\AP} with $\bar{\bm{c}}=(\bm{0}_{1\times 5},\bm{1}_{1\times 8})$ and $\bar{ \bm{d}}=(\bm{1}_{1\times 8},\bm{1}_{1\times 5})$ by Lemma \ref{lem:lifting}.
Actually it is equivalent to the optimal solution $\bar{\bm{c}}=(\bm{0}_{1\times 5},1,1,0,1,1,0,1,1)$ and $\bar{ \bm{d}}=(\bm{1}_{1\times 8},1,0,0,1,1)$  calculated by the $\Sigma$-method and
both give the same optimal value of the new system $\bar {\delta} =  \delta-n+r = 8-8+5$.

Then we can verify that the determinant of the new Jacobian matrix is a non-zero constant. Furthermore, the embedding method in this example finishes the index reduction after one iteration, rather than $3$ iterations by the substitution method or the augmentation method shown in section $6.2$ \cite{Taihei19}. So the embedding method is more efficient for this example.

Specifically, for the numerical solution of this problem, the initial value of $\bm{u}(0)$ and $\bm{\xi}$ are set corresponding to $\dot{\bm x}(0)$ respectively. 

\subsection{Non-linearly Modified Pendulum}\label{ssec:exam2}
%

This nonlinear {\DAE} system consisting of $4$
differential equations and $1$ algebraic equation is obtained by dynamic analysis and modeling of a simple pendulum. See \cite{Mazzia08} for more details. The system is:

\[\left\{\begin{array}{rcc}
\dot{x}_{4}-{x}_{1}\cdot{x}_{2}\cdot\cos(x_{3})&=&0\\
\dot{x}_{5}-{x}_{2}^{2}\cdot\cos({x}_{3})\cdot\sin(x_{3})+g&=&0\\
{x}_{1}^{2}+{x}_{2}^{2}\cdot\sin({x}_{3})^{2}-1&=&0\\
\tanh((\dot{x}_{1}-x_{4}))&=&0\\
\dot{x}_{2}\cdot\sin(x_{3})+x_{2}\cdot\dot{x}_{3}\cdot\cos(x_{3})-x_{5}&=&0
\end{array}\right.\]

After structural analysis, we get the dual optimal solution $\bm{c}=(0,0,1,0,0)$ and $\bm{d}=(1,\cdots,1)$, with $\delta=4$ and $n=5$. Moreover, the rank of the Jacobian matrix is $\rank \bm{\Jac} = r =\rank \bm{\Jac}[(3,2,1,4),(3,1,4,5)]= 4$. Thus, the constraint is $\bm{F}^{(\bm{c}-1)}=\{{x}_{1}^{2}+{x}_{2}^{2}\cdot\sin({x}_{3})^{2}-1=0\}$.

From Section {\ref{ssec:exam1}}, by the embedding method, let $\bm{s}=\{\dot{x}_{3}, \dot{x}_{1}, \dot{x}_{4}, \dot{x}_{5}\}$, $\bm{y}=\{\dot{x}_{2}\}$, $\bm{f(s,y,z)}=\{F_{3}, F_2, F_1, F_4\}$ and $\bm{g(s,y,z)}=\{ F_5\}$.  Then we need to replace $\bm{s}$  by $\{u_1, u_2, u_3, u_4\}$ and $\bm{y}$ by a constant $\xi \in \R$ in $\hat{\bm{F}}$, respectively. Finally, we can get a modified {\DAE} $\{\bm{F}^{(\bm{c}-1)}, \bm{F}^{\rm aug}\}$, with $\bm{F}^{\rm aug} = \{\bm{f(s,y,z)},\hat{\bm{F}}\}$.


We can construct a new optimal solution $(\bar{\bm{c}},\bar{\bm{d}})$ of Problem (\ref{LPP}) of $\bm{F}^{\rm aug}$ by Lemma {\ref{lem:lifting}} directly, which yields ${\bm{c}}=(\bm{0}_{1\times 4},0,0,\bm{1}_{1\times 3})$ and ${\bm{d}}=(\bm{1}_{1\times 5},1,1,0,0)$ with the same optimal value $\bar {\delta} = \sum{\bar{d}_{j}}-\sum{\bar{c}_{i}}-\#eqns(\bm{F}^{(\bm{c}-1)})=9-5-1=\delta-n+r=4-5+4$. 


Unfortunately, the Jacobian matrix of the new top block $ \bm{F}^{\rm aug} $ is still singular, with $\rank \bm{\Jac}(\bm{F}^{\rm aug}) =\rank \bm{\Jac}[(1:6,8:9),(1,3:9)] = 8$. Similarly, we need another modification of $\bm{F}^{\rm aug}$ by the embedding method. Finally, this {\DAE} system has been regularized.  The final optimal value is $2 = \bar{\delta}-9+8$.

Compared with the one additional equation of the augmentation method, the embedding method in this example will introduce more equations which will affect the efficiency of the numerical solution, although both methods can be successful after two steps of regularization. However, this adverse effect only exists when the Jacobian matrix is very close to being full rank, {\ie} $r=n-1$.


\subsection{Analysis of Bending Deformation of a Beam}
The specific description is given in Example \ref{ex:3}. In this example, when the input energies of the beam are the same, we can set $\lambda = 1$. Specifically, in order to to check the correctness of our numerical solution of the  global structural analysis method, we set the measured sum of loads $(q_1(x)+q_2(x))=-\frac{EI}{5\cdot x^2}\cdot((1-\sin(x))+y_{1}(x))$. Thus, two exact solutions of this {\DAE} are
 \begin{eqnarray*}
  y_{1}(x) &=& +y_{2}(x)=C_1\cdot\sin(\frac{\sqrt{2}x}{2})+C_2\cdot\cos(\frac{\sqrt{2}x}{2})-\frac{1}{5}\cdot\sin(x)-\frac{1}{5} \\
   y_{1}(x)&=& -y_{2}(x)=-\frac{1}{5}\cdot(1-\sin(x))
 \end{eqnarray*}
 Here $C_1$ and $C_2$ are constants depending on consistent initial conditions.

By structural analysis, the optimal solutions is $\bm{c}=(0,2)$ and $\bm{d}=(2,2)$.

 This non-linear {\DAE} has two components of constraints resulting from its constraints: one component of constraints results from $y_{1}=y_{2}$, the other component of constraints results from $y_{1}=-y_{2}$. In detail, two witness points
 are computed by the Homotopy continuation method \cite{WWX2017} where each point has
 coordinates $(y_{1}, y_{2}, \dot{y}_{1},\dot{y}_{2})$:
\[
\begin{array}{rrrrrr}
  ( & -0.43092053722 &  -0.43092060160 &  -0.27565041470 & -0.27565030340 & ) \\
  ( & -0.19993949748 &  +0.19993723792 &  +0.64332968577 & -0.64333747822 & ) \\
\end{array}
\]
 Because the Jacobian matrix of the polynomial constraints is singular here, a large penalty factor should be introduced in order to improve convergence.  These witness points are approximate points near the consistent initial value points, which and need to be refined by Newton iteration.


 Obviously, the Jacobian matrix is non-singular for any witness point from the component of constraints with $y_{1}=y_{2}$. This case can be solved directly after applying the $\Sigma$-method. On the contrary, for any witness point on the component of constraints with $y_{1}=-y_{2}$, the Jacobian matrix will degenerate to a singular matrix. For this case, we have to construct its equivalent {\DAE}.


\subsection{Result Analysis}
We use Matlab R$2021$a for the numerical computations with the error settings AbsTol =$10^{-6}$ and RelTol = $10^{-3}$. And MATLAB's ode$15$i is used as the ode solver of algorithm for numerical solution of the examples in our paper. The numerical experiment results by the embedding method are shown in Figure {\ref{fig:ex}}, which illustrates the correctness of our methods proposed in this paper.

\begin{figure}[htbp]
\centering
\subfigure[Transistor
Amplifier]{
\includegraphics[width=0.40\textwidth,height=0.25\textwidth]{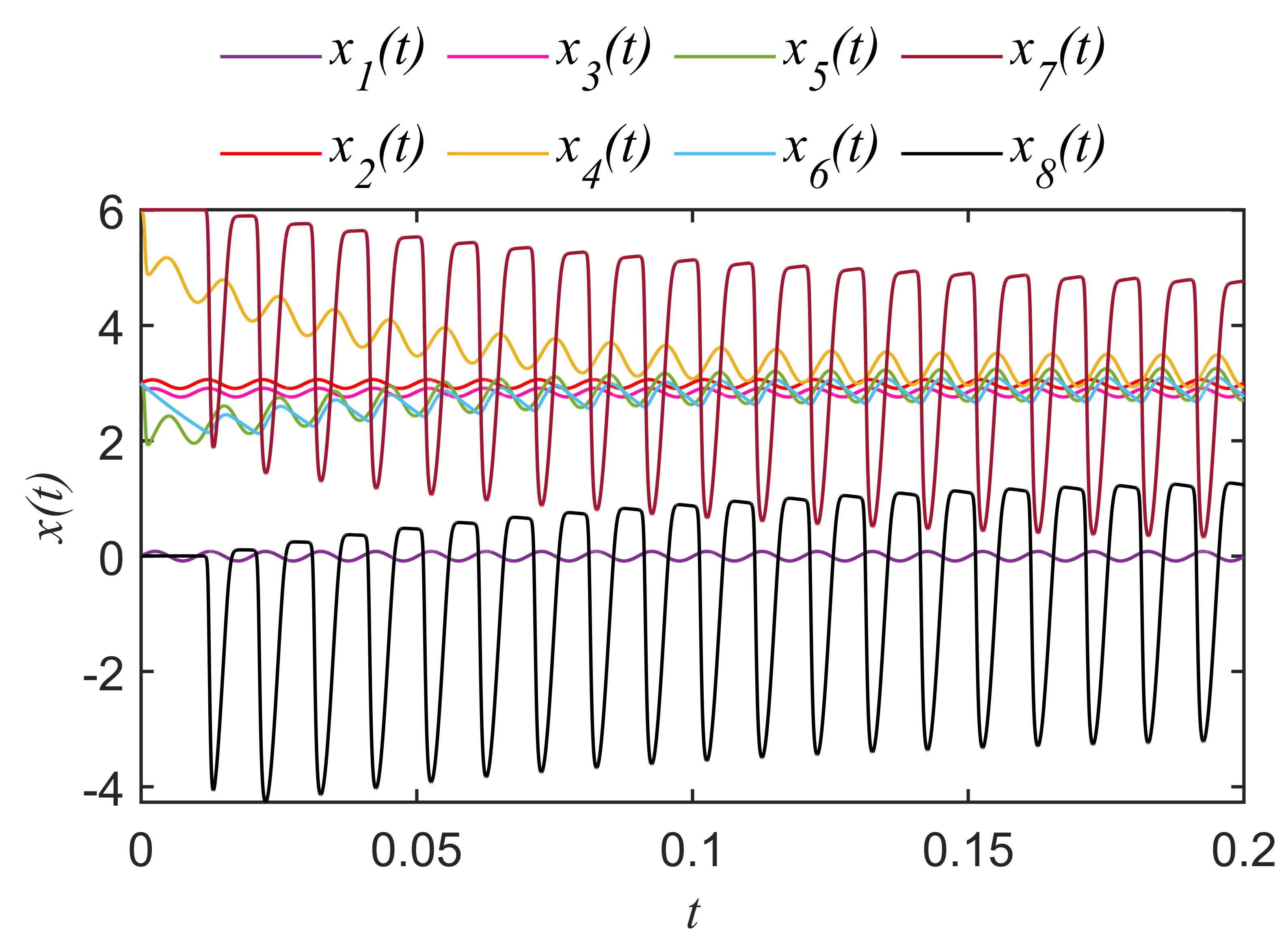}
}
\quad
\subfigure[Modified
Pendulum]{
\includegraphics[width=0.40\textwidth,height=0.25\textwidth]{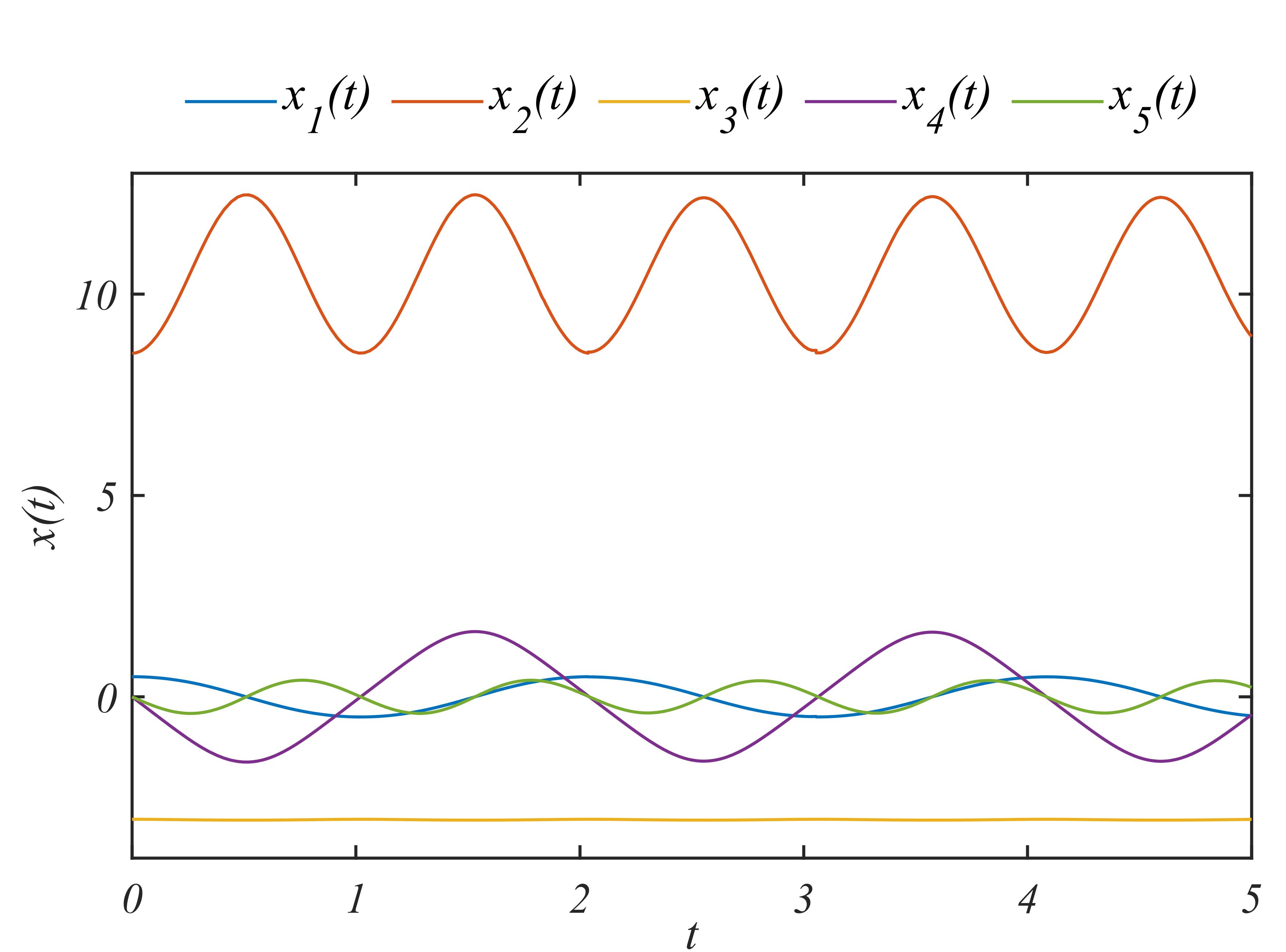}
}
\quad
\subfigure[Beam (non-singular)]{
\includegraphics[width=0.40\textwidth,height=0.25\textwidth]{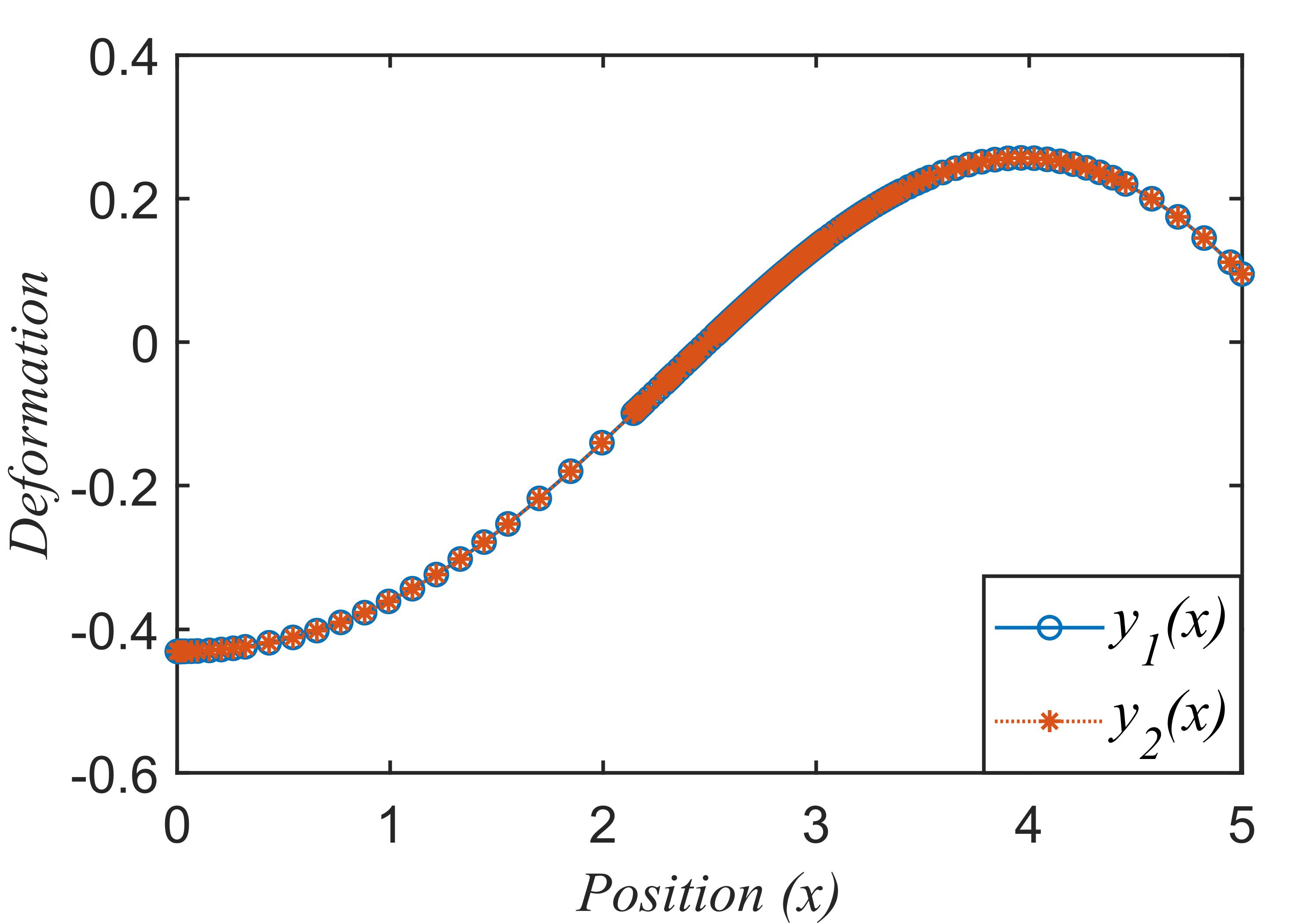}
}
\quad
\subfigure[Beam (singular)]{
\includegraphics[width=0.40\textwidth,height=0.25\textwidth]{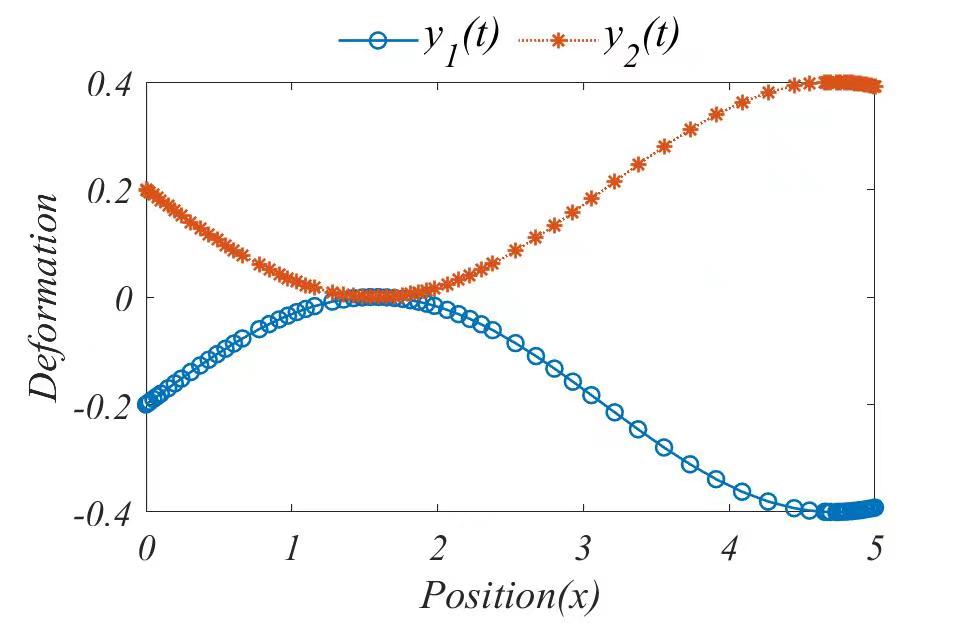}
}
\caption{Numerical Solution of $3$ Examples by the Embedding Method}\label{fig:ex}
\end{figure}

The experimental results are summarized in Table \ref{tab:result}, which show that the substitution method and the augmentation method are effective in dealing with symbolic cancellation {\DAE}s. Without regularization, the $\Sigma$-method only works well on S-amenable cases.
The former three methods all fail in Example \ref{ex:3}, as they cannot detect that the Jacobian matrix has been constrained to be singular. Furthermore, the Homotopy continuation method used in the embedding method helps to detect numerical degeneration by computing a constant rank of Jacobian matrix at witness points. The global structural analysis method based on the embedding method can deal with symbolic cancellation {\DAE}s and numerically degeneration {\DAE}s.

\begin{table}[htpb]
	\caption{Comparison of Experimental Results.
Index = the index reported by the $\Sigma$-method initially, \hskip2pt\#C = \#Components of Constraints, \hskip2pt $\surd$ = success, \hskip2pt $\times$ = failure, \hskip2pt * = possible failure depending on Jacobian matrix, \hskip2pt Sub = Substitution \cite{Taihei19}, \hskip2pt Aug = Augmentation \cite{Taihei19}, \hskip2pt Em = Embedding, \hskip2pt DE = Degeneration, \hskip2pt SC = symbolic cancellation \hskip2pt ND = numerical degeneration. }\label{tab:result}
\centering
\begin{tabular}{|c|c|c|p{50pt}|c|c|c|c|p{20pt}|}
  \hline
  Index & Structure & \#C & Examples & $\Sigma$ \cite{Nedialkov2008} & Sub & Aug & Em & DE\\
  \hline
  1 & linear& $1$ & Transistor Amplifier &$\times$ & $\surd$ &   $\surd$ &  $\surd$ & SC\\
   \hline
  3 & non-linear& $1$ &Modified  Pendulum & $\times$& $\surd$ &  $\surd$ &   $\surd$& SC\\
   \hline
  2 & non-linear& $2$ &Beam & * & * & * &  $\surd$ & ND\\
  \hline
\end{tabular}
\end{table}


\section{Conclusions}\label{sec:con}

We first generalized the definition of optimal value to {\DAE}s with block structure in Section \ref{s:pre}.
In Section \ref{sec:implictit method}, we proposed an improved structural method --- the embedding method ---  based on degeneration detection described in Section \ref{sec:det}. Then we presented a global structural analysis method in Section \ref{sec:poly}. We illustrated our approach by $3$ examples in Section \ref{sec:ex}.

By introducing new variables and equations to increase the dimension of the space in which a {\DAE} resides, the embedding method avoids the direct elimination of the {\DAE}. Under certain conditions, it avoids solving
assignment problems for the new systems. The embedding method is an extension of the augmentation method. It is efficient and intuitive, and enables the simultaneous regularization of the whole {\DAE}, rather than one specific equation at a time. The more rank deficiency encountered, the higher our method's efficiency, but the equation's scale will also increase. Unlike other regularization methods only work in the neighborhood of a consistent point, an S-unamenable {\DAE} can be regularized on a whole component of constraints (except the singular set in Remark \ref{rem:constant_rank}) based on Lemma \ref{lem:whole}. 

 A strong feature of the global structural analysis method, is that Homotopy continuation methods can be naturally and efficiently combined with the embedding method, which can help to deal with almost all degeneration issues for polynomial {\DAE}s. However, existing methods are difficult to obtain consistent points on each component of constraints for {\DAE}s with transcendental equations or strong non-linearity in applications. In this paper, we have established the theoretical foundation of the global structural analysis method for both polynomial systems and analytic systems, but finding all solutions of an analytic system under certain conditions is a worthwhile research topic in the future work.
 

\medskip\noindent{\bf \normalsize Acknowledgements.}
This work is partially supported by the National Key Research Project of China under Grant No.2023YFA1009402, the projects of Chongqing (2021000263, cstc2021yszx-jcyjX0004, cstc2021jcyj-msxmX0821) and the National Natural Science Foundation of China (12301650). 

\begin{appendices}
\section{Algorithms}\label{sec:alg}
\sloppy{}

The global structural analysis method (Algorithm \ref{alg:7}) for solving a polynomial
{\DAE}, is based on the embedding method (Algorithm \ref{alg:6}) --- the key algorithm to restore the regularity. Also, we need to recall some existing subroutines given in Algorithms $1-5$.

Algorithm $1$ is used to find an optimal solution $(\bm{c},\bm{d})$ of Problem (\ref{LPP}) of a {\DAE} ${\bm F}$ with variables $\bm{x}$, which helps to differentiate {\DAE} in a special pattern to reduce its differential index.

Algorithm $2$ is used to find a real witness set $\bm{W}=\{\bm{p}_{i}|i=1,...,m\}$ by the homotopy continuation method. Here, the input $\bm{f}$ is considered as a polynomial system by taking all derivatives of $\bm x$ as new variables. For constraints of a {\DAE}, the obtained real witness points can be considered as candidate initial points. Crucially, this algorithm can find all components of constraints of a {\DAE}.


Algorithm $3$ is a sorting method to find a sub-matrix with constant rank
by swapping the equations of the top block $\bm{B}_{k_{c}}$ and the highest derivative variables $\bm{X}^{k_d}$.
The output is a new $\bm{B}_{k_{c}}$ whose Jacobian matrix at a given real witness point $\bm{p}$ has a full rank sub-matrix $\bm{\Jac}(\bm{p})[1\:r,1\:r]$, where $r$ is determined by Singular Value Decomposition (SVD) \cite{Golub13}. Firstly, calculate permutation vectors of rows and columns for $\bm{\Jac}(\bm{p})$ respectively by Householder QR (HQR). Then, swap equations and variables according to permutation vectors respectively. Before returning the sorted matrix, we will verify the rank of $\bm{\Jac}(\bm{p})[1\:r,1\:r]$ by SVD.

Algorithm $4$ is a low index {\DAE} solver implemented by one-step projection and one-step prediction. Obviously a low index {\DAE} $\bm{F}^{(\bm{c})}$ can be divided into two parts --- constraints $\bm{F}^{(\bm{c-1})}$ and a square {\ODE} $\bm{B}_{k_{c}}$. Firstly, since an initial value may be not a consistent initial value of the {\ODE}, the initial value point needs to  be projected back onto the constraints by Newton iteration to find a nearby consistent initial value point satisfying the constraints. Secondly, an {\ODE} solver,
such as the Runge-Kutta method or the Euler method, is used to make a one-step prediction from the previous consistent initial value point. Through step-by-step iteration, the {\DAE} can be solved numerically, where the tolerance can be set as needed.

\begin{breakablealgorithm}
\caption{}
 \begin{algorithmic}[1]
  \State  $(\bm{c},\bm{d})=Structure(\bm{F}, \bm{x})$, such as the $\Sigma$-method see \cite{Nedialkov2008}.
   \end{algorithmic}
\end{breakablealgorithm}

\vspace{0.2cm}

\begin{breakablealgorithm}
\caption{}
 \begin{algorithmic}[1]
  \State  $\bm{W}=\{\bm{p}_{i}|i=1,...,m\}=witness(\bm{f})$, see \cite{WRF17}. // $m$ is the number of real witness points.
   \end{algorithmic}
\end{breakablealgorithm}

\vspace{0.2cm}


\vspace{0.2cm}

\begin{breakablealgorithm}
\caption{}\label{alg:4}
 \begin{algorithmic}[1]
 \Require the top block equations $\bm{B}_{k_{c}}$, Jacobian matrix $\bm{\Jac}$ with witness point $\bm{p}$, the constant rank $r$,  absolute tolerance $AbsTol$
 \Ensure  recombination of the top block equations $\bm{B}_{k_{c}}$
 \Function {SORT}{$\bm{B}_{k_{c}}, \bm{\Jac}(\bm{p}), r, AbsTol$}
  \State $\bm{piv_{row}}=HQR(\bm{\Jac}(\bm{p}),AbsTol)$, see Section $5.2$  \cite{Golub13};
  \State $\bm{piv_{col}}=HQR(\bm{\Jac}^{T}(\bm{p}),AbsTol)$; // $\bm{piv_{row}}$ and $\bm{piv_{col}}$ are the permutation vector of rows and columns, respectively;
  \State $\bm{B}_{k_{c}}=\bm{B}_{k_{c}}[\bm{piv_{row}}]$, // swap equations;
  \State $\bm{B}_{k_{c}}=\bm{B}_{k_{c}}(\bm{X}_{k_d}[\bm{piv_{col}}])$, // swap the highest derivative variables;
  \State verify the rank of $\bm{\Jac}(\bm{p})[1\:r,1\:r]$ by SVD.
   \EndFunction
   \end{algorithmic}
\end{breakablealgorithm}

\vspace{0.2cm}

\begin{breakablealgorithm}
\caption{}
 \begin{algorithmic}[1]
 \Require low index {\DAE} equations $\bm{F}^{(\bm{c})}$ and dependent variables $\bm{x}$ with independent variable $t\in [t_0,t_{end}]$, initial point $\bm{p}$,  absolute tolerance $AbsTol$ and relative tolerance $RelTol$
 \Ensure  numerical solutions of {\DAE} $\bm{x}(t)$
 \Function {{DAESOLVER}}{$\bm{F}^{(\bm{c})}, \bm{x}, \bm{p}, [t_0, t_{end}], AbsTol, RelTol$} //
  \State  $j=0$, $\bm{x}(t_{0})=\bm{p}$, set step $h$ and the maximum number of iterations $N$;
  \While {$t_{j}<=t_{end}$}
   \State $\bm{x}(t_{j}) = Newton(\bm{F}^{(\bm{c-1})},\bm{x}(t_{j}), AbsTol, N)$ // Refinement, see \cite{Mathews2004};
  \State $\bm{x}(t_{j+1}) = OdeSolver(\bm{B}_{k_c}, \bm{x}(t_{j}), AbsTol, RelTol, h)$ // such as $ode45$, Euler method, $ode15i$ {\etc};
  \State $j=j+1$;
   \EndWhile
   \EndFunction
   \end{algorithmic}
\end{breakablealgorithm}

\vspace{0.2cm}

\renewcommand{\thealgorithm}{5}\label{alg:alg1}
\begin{breakablealgorithm}
\caption{The Embedding Method}\label{alg:6}
 \begin{algorithmic}[1]
    \Require {\DAE} equations $\bm{F}$ and dependent variables $\bm{x}$ with independent variable $t\in [t_0,t_{end}]$, real witness point $\bm{p}$,  absolute tolerance $AbsTol$
    \Ensure  modified {\DAE} new equations $\bm{F}^{(\bm{c})}$ and new real witness point $\bm{p}$

    \Function {Embedding}{$\bm{F},\bm{x},\bm{p}_{i}, AbsTol$}
    \While {true}
    \State Structural Analysis: $(\bm{c},\bm{d})=Structure(\bm{F},\bm{x})$
    \State $n=length(\bm{c})$, $k_d = \max(d_j)$, $k_c = \max c_i$, $\delta=\sum d_j - \sum c_i$ \label{alg1goto}
    \State Construct: $\bm{F}^{(\bm{c})}$  , $\bm{B}_{k_{c}}$,  $\bm{\Jac}$ by Equations (\ref{eq:DefFc},\ref{eq:B_i})
    \State Compute the rank of $\bm{\Jac}(\bm{p})$ by SVD \cite{Golub13}
    \If{$r=n$}
    \State \Return  $\bm{F}^{(\bm{c})}$, $\bm{p}$
    \ElsIf {$\delta-(n-r) \leq 0$}
    \State \Return Error // this {\DAE} does not have a solution.
    \EndIf
    \State $\{\bm{f(s,y,z)},\bm{g(s,y,z)}\}=$SORT$(\bm{B}_{k_{c}}, \bm{\Jac}(\bm{p}), r, AbsTol)$
    \State Introduce $n$ new equations $\hat{\bm{F}}=\{\bm{f(s,y,z)},\bm{g(s,y,z)}\}$
    \State Replace $\bm{s}$  by $\bm{u}$ in $\hat{\bm{F}}$ // refer to Definition {\ref{define_IRE}}
     \State Replace $\bm{y}$ by constants $\bm{\xi} \in \R^{n-r}$ in $\hat{\bm{F}}$
     \State Substitute $\{t_0, \bm{p}, \bm{\xi}\}$ into $\bm{f(s,y,z)}$ to calculate $\bm{s}$, note as $\hat{\bm{u}}$
    \State  $\bm{p}\leftarrow(\bm{p}, \hat{\bm{u}})$ //  corresponding lifting
of consistent initial value
     \State $\bm{F}^{\rm aug} = \{\bm{f(s,y,z)},\hat{\bm{F}}\}$
     \State $\bm{F} \leftarrow \{ \bm{F}^{(\bm{c}-1)}, \bm{F}^{\rm aug}\}$, $\bm{x} \leftarrow (\bm{x},\bm{u})$ //extend equations and variables
     \If {$\bm{F}^{\rm aug}$ satisfies Lemma {\ref{lem:lifting}}}
     \State $\bar{\bm{c}}= [\bm{0}_{r}, \bm{1}_{n}]$, $\bar{\bm{d}}= [\bm{d},\bm{1}_{r}]$ by Equation (\ref{opt_soln})
     \State  $\bm{c}=[\bm{0}_{(\sum c_{j})},\bar{\bm{c}}]$, $\bm{d}=\bar{\bm{d}}$
      \State Goto \ref{alg1goto}
     \EndIf
    \EndWhile
    \EndFunction
 \end{algorithmic}
\end{breakablealgorithm}

\renewcommand{\thealgorithm}{6}\label{alg:alg2}
\begin{breakablealgorithm}
\caption{Global Structural Analysis Method}\label{alg:7}
 \begin{algorithmic}[1]
    \Require {\DAE} equations $\bm{F}$ and dependent variables $\bm{x}$ with independent variable $t\in [t_0,t_{end}]$,  absolute tolerance $AbsTol$ and relative tolerance $RelTol$
    \Ensure numerical solutions of {\DAE} $\bm{x}^{*}(t)$
    \State Initialization:  check the number of equations \#eqns and dependent variables \#dvars
    \If {\#eqns $\neq$ \#dvars} \State \Return False \EndIf
    \State Set $\bm{x}^{*}(t)=\{\}$
    \State Structural Analysis: $(\bm{c},\bm{d})=Structure(\bm{F},\bm{x})$
    \State Construct: differentiated system $\bm{F}^{(\bm{c})}$,  Jacobian matrix $\bm{\Jac}$
    \State Find real witness points: $ \bm{W}=\{ \bm{p}_{i}=witness(\bm{F}^{(\bm{c-1})}(t_0))|i=1,...,m\}$  //$m$ is number of real witness points.
    \For {$\bm{p}_i\in \bm{P}$}
    \State $\{\tilde{\bm{F}},\tilde{\bm{p}}_{i}\}:=Embedding(\bm{F},\bm{x},\bm{p}_{i}, AbsTol)$
    \State $\tilde{\bm{x}}(t)=$DAESOLVER$(\tilde{\bm{F}},\bm{x},\tilde{\bm{p}}_{i}, [t_0, t_{end}], AbsTol, RelTol)$
    \State $\bm{x}^{*}(t)=\{\bm{x}^{*}(t), \tilde{\bm{x}}(t)[1,...,n]\}$
    \EndFor
    \State \Return $\bm{x}^{*}(t)$
 \end{algorithmic}
\end{breakablealgorithm}
\end{appendices}

\bibliography{sn-bibliography}
\end{document}